\documentclass[11pt]{amsart}
\usepackage{amsmath, amssymb, amscd, cancel, graphicx, soul, stmaryrd, accents, bbm}
\pdfoutput=1
\usepackage[colorlinks=true, urlcolor=NavyBlue, linkcolor=NavyBlue, citecolor=NavyBlue, bookmarks=false]{hyperref}

\usepackage[margin=1in]{geometry}
\usepackage[dvipsnames,usenames]{xcolor}
\usepackage{mathdots}
\usepackage{float}
\usepackage{amsthm}
\usepackage{amsfonts}
\usepackage[parfill]{parskip} 
\usepackage{paralist}
\usepackage{marginnote}
\usepackage[all]{xy}
\numberwithin{equation}{section}
\usepackage{tikz}
\usepackage{wrapfig}
\usepackage{pinlabel}
\usepackage{multirow}
\usepackage{blkarray}
\usepackage{mathtools}
\usepackage{enumitem}

\usepackage{lmodern}
\usepackage[T2A,T1]{fontenc} 
\usepackage[utf8]{inputenc}
\usepackage{inputenc}
\usepackage[UKenglish]{babel}
\usepackage{yfonts}
\usepackage{bm}
\usepackage{afterpage}
\usepackage{ifthen}
\usepackage{tabularx}
\usepackage{nicefrac}
\usepackage{picins}
\usepackage{soul}
\usepackage{calc}
\usepackage{float}
\usepackage{wrapfig}

\usepackage{picins}
\usepackage[font=footnotesize,labelfont=bf]{caption}
\usepackage{wrapfig}
\DeclareMathAlphabet{\mathpzc}{OT1}{pzc}{m}{it}

\usepackage{caption}
\usepackage[labelfont=bf]{caption}
\DeclareCaptionFormat{myformat}{\fontsize{9}{10}\selectfont#1#2#3}
\usepackage{placeins}
\usepackage{pinlabel}

\usepackage[nameinlink]{cleveref}

\newtheorem{thm}{Theorem}[section]
\newtheorem{conj}[thm]{Conjecture}
\newtheorem{cor}[thm]{Corollary}
\newtheorem{lemma}[thm]{Lemma}
\newtheorem{prop}[thm]{Proposition}

\newtheorem{defn}[thm]{Definition}
\newtheorem{rmk}[thm]{Remark}

\newtheorem{claim}[]{Claim}

\newcommand{\Z}{\mathbb Z}

\newcommand{\sgn}{\text{sign}}
\newcommand{\tr}{\text{tr}}

\begin{document}

\title[Twist Positivity, L-space knots, and Concordance]
{Twist Positivity, L-space knots, and Concordance}

\author{Siddhi Krishna and Hugh Morton}

\address{Department of Mathematics, Columbia University\\ New York, NY, United States}
\email{sk5026@columbia.edu}

\address{Department of Mathematica Sciences, University of Liverpool\\Liverpool, United Kingdom}
\email{morton@liv.ac.uk}


\maketitle

\medskip

\noindent \small{{\bf Abstract.} 
Many well studied knots can be realized as positive braid knots where the braid word contains a positive full twist; we say that such knots are \textit{twist positive}. Some important families of knots are twist positive, including torus knots, 1-bridge braids, algebraic knots, and Lorenz knots. We prove that if a knot is twist positive, the braid index appears as the third exponent in its Alexander polynomial. We provide a few applications of this result. After observing that most known examples of L-space knots are twist positive, we prove: if $K$ is a twist positive L-space knot, the braid index and bridge index of $K$ agree. This allows us to provide evidence for Baker's reinterpretation of the \textit{slice-ribbon conjecture}: that every smooth concordance class contains at most one fibered, strongly quasipositive knot. In particular, we provide the first example of an infinite family of positive braid knots which are distinct in concordance, and where, as $g \to \infty$, the number of hyperbolic knots of genus $g$ gets arbitrarily large. Finally, we collect some evidence for a few new conjectures, including the following: the braid and bridge indices agree for any L-space knot. \\
\tiny

}
 
\normalsize 

\section{Introduction} \label{section:intro}


By Alexander's theorem \cite{Alexander}, every link in $S^3$ can be realized as the closure of some braid. Some properties of braids can be ported into important topological properties of their closures. For example, if a braid $\beta$ is \textit{homogeneous} (i.e. for every $i$, every $\sigma_i$ appearing in $\beta$ has the same sign), then the closure is a fibered link in $S^3$ \cite{Stallings:Fibered}. In a different vein, the closures of \textit{quasipositive braids} are exactly the links arising as the transverse intersections of smooth algebraic curves $f^{-1}(0) \subset \mathbb{C}^2$ (here, $f$ is a non-constant polynomial) with the unit sphere $S^3 \subset \mathbb{C}^2$ \cite{Rudolph:QP, BoileauOrevkov, Hayden:QP}. Another property of braids with important topological consequences is \textit{twist positivity}:

\begin{defn}
Let $\beta$ be a positive braid braid on $n$ strands. If $\beta$ can be factored to contain a positive full twist on $n$ strands $($i.e. $\beta = \Delta^2 \gamma$ where $\gamma$ is a positive braid word$)$, then $\beta$ is a \textbf{twist positive} braid on $n$ strands. If $K$ can be realized as the closure of a twist positive braid on $n$ strands, then $K$ is a \textbf{twist positive} knot on $n$ strands.
\end{defn}

\begin{wrapfigure}{R}{.15\linewidth}
    \centering
    \vspace{-0.8cm}
    \includegraphics[scale=.18]{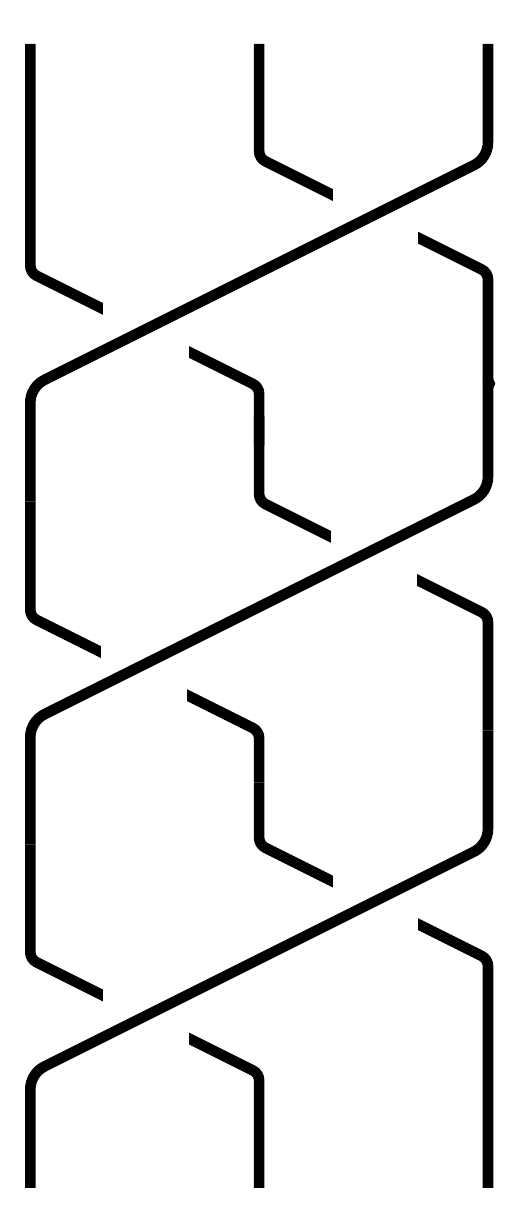}
    \captionof{figure}{}
    \label{fig:FullTwist}
\end{wrapfigure}

A full twist on three strands is seen in \Cref{fig:FullTwist}; here, $\Delta^2 = (\sigma_2 \sigma_1)^3$. Twist positivity is a well studied property: perhaps most notably, Morton and Franks-Williams proved that if $K$ is twist positive on $n$ strands, then the braid index of $K$ is $n$ \cite{Morton:KnotPolys,FranksWilliams}. 

In this paper, we show that twist positivity unifies classical results in knot theory with modern concordance invariants. Our first theorem is about the Alexander polynomial of twist positive knots. 

\begin{thm} \label{thm:AlexPoly}
Suppose $K$ is twist positive on $n$ strands. Then the Alexander polynomial of $K$ has the form $\Delta_K(t) = 1 - t + t^n +t^{n+1}R(t)$, where $R(t) \in \Z[t]$. In particular, the braid index of $K$ appears as an exponent in $\Delta_K(t)$. 
\end{thm}

The strength and utility of \Cref{thm:AlexPoly} is leveraged by way of the connection between the Alexander polynomial and knot Floer homology: not only does knot Floer homology categorify the Alexander polynomial, but it also provides a suite of powerful knot invariants. We recall: an \textit{L-space knot} in $S^3$ is a non-trivial knot which admits a positive Dehn surgery to an \textit{L-space}, the closed, connected, oriented 3-manifolds with simple Heegaard Floer homology. For L-space knots, both the Alexander polynomial and knot Floer homology take on a restricted form \cite[Theorem 1.2]{OSz:LensSpaceSurgeries}. This connection leads to our first application of \Cref{thm:AlexPoly}. 

\begin{thm} \label{thm:Lorenz}
If $K$ is a twist positive L-space knot, then its braid index and bridge index agree. Moreover, $K$ realizes the minimal braid and bridge indices of its concordance class. 
\end{thm}

\begin{rmk}
As we observe in \Cref{section:Lorenz}, most explicit families of L-space knots are not only twist positive, but they also happen to be Lorenz knots. 
\end{rmk}

Lorenz knots exhibit rich phenomena from the dynamical, geometric, and braid theoretic perspectives. They have been intensely studied since their introduction in 1963; see, for example, \cite{BirmanWilliams, BirmanKofman, ChampanerkarFuterKofmanNeumannPurcell, Dehornoy:Lorenz, dePaivaPurcell}, or \cite{Birman:Lorenz} for a survey. In \Cref{section:Lorenz}, we observe L-space Lorenz knots include the well-studied \textit{algebraic knots}, \textit{1-bridge braids} \cite{Gabai:1BridgeBraids, GLV:11LSpace}, and some \textit{twisted torus knots} \cite{Vafaee:TwistedTorusKnots}, amongst others.

Therefore, \Cref{thm:Lorenz} generalizes results of both Juh\'asz-Miller-Zemke \cite{JuhaszMillerZemke} (who prove \Cref{thm:Lorenz} for torus knots) and Hom-Lidman-Park \cite{HomLidmanPark}, who use input from bordered Floer homology to prove it for algebraic knots. Our results suggests that twist positivity is the important topological property guiding this behavior.

Birman-Williams \cite[Section 5]{BirmanWilliams} noticed that Lorenz knots are always twist positive on some number of strands, and conjectured that this quantity was actually the braid index. Indeed, this conjecture follows with the subsequent announcement of the Franks-Williams result \cite{FranksWilliams}. To the best of our knowledge, the bridge index of Lorenz knots has not been investigated. \Cref{thm:Lorenz} proves that often, the bridge index and the braid index for a Lorenz knot agree, hence the Lorenz presentation cannot be minimized further to decrease the bridge index.

\Cref{thm:AlexPoly} and \Cref{thm:Lorenz} suggest that twist positivity is not only an important 3-dimensional property, but that it also provides some insights into 3.5- and 4-dimensional phenomena. Indeed, we next study concordance. Baker \cite{Baker:SliceRibbon} observed a novel relationship between concordance properties of fibered knots and Fox's \textit{slice-ribbon conjecture}, which asks whether every slice knot admits a ribbon presentation in $S^3$. In particular, Baker conjectured that \textit{if two fibered, strongly quasipositive knots are concordant, then they are isotopic}. Baker observes that either his conjecture is true, or the Slice-Ribbon conjecture is false\footnote{Baker phrases his conjecture in terms of \textit{tight fibered knots}, but by \cite{Hedden:Positivity}, this is equivalent to studying fibered, strongly quasipositive knots.}. Therefore, studying the concordance properties of fibered, strongly quasipositive knots is of central interest. L-space knots form a prominent class of fibered, strongly quasipositive knots. As noted by Abe-Tagami \cite{AbeTagami}, if Baker's conjecture is true, then every concordance class should have at most one L-space knot.

\begin{thm}\label{thm:DistinctConcordanceClasses}
Let $\mathcal{S} = \mathcal{T} \cup \mathcal{L}$, where $\mathcal{T}$ is the set of positive torus knots, and $\mathcal{L}$ is the set of L-space knots of braid index three. Every concordance class contains at most one knot from $\mathcal{S}$. 
\end{thm}

\begin{rmk}
Our proof of \Cref{thm:DistinctConcordanceClasses} is independent of Litherland's proof that torus knots are linearly independent in concordance \cite{Litherland:TorusKnots}. 
\end{rmk}

\Cref{thm:DistinctConcordanceClasses} is proved by using the braid index, $\tau(K)$, and the Seifert genus in concert.

Baader-Dehorney-Liechti \cite{BaaderDehornoyLiechti} proved that only finitely many positive braid knots can be concordant to each other (c.f. \cite{Stoimenov:ConcordancePositiveBraids}). However, to the best of our knowledge, we produce the first infinite family of hyperbolic such examples.

\begin{cor} \label{cor:InfinitelyDistinctConcordance}
There is an explicit infinite family of positive braid knots that are distinct in concordance, where as $g \to \infty$, the number of \textit{hyperbolic} knots of genus $g$ gets arbitrarily large.
\end{cor}

While L-space knots are well-studied, a classification remains elusive; nevertheless, collecting their properties remains a central goal. L-space knots are ``simple'' from many perspectives: they have simplest possible knot Floer homology \cite{OSz:LensSpaceSurgeries}, ``simple'' knot exteriors (their exteriors are realized as mapping tori) \cite{Ghiggini:Fibered, YiNi:FiberedKnots}, and, heuristically, they seem to have ``simple'' monodromy \cite{MisevSpano, Ni:HFKmonodromy, GhigginiSpano}. The bridge index and the braid index provide different measures of complexity of a knot, as the former is a lower bound on the latter. The difference between these quantities is itself a measure of complexity. We conjecture that this defect is trivial for L-space knots. 

\begin{conj} \label{LspaceIndex}
If $K$ is an L-space knot, then its braid index and bridge index agree.
\end{conj}


This conjecture is true for the most well known families of L-space knots: 

\begin{itemize}
\item The only L-space two-bridge knots are $T(2,n)$ torus knots. \cite{Goodrick:2BridgeAlternating, OSz:LensSpaceSurgeries}. 
\item The only L-space Montesinos knots are the $P(-2,3,q)$, $q \geq 1$ pretzel knots \cite{LidmanMoore:PretzelKnots, BakerMoore:Montesinos}. They are all twist positive on either two or three strands, hence \Cref{thm:Lorenz} applies.
\item As proved in this article, the braid and bridge index agree for \textit{1-bridge braids}. 
\item Baker-Kegel \cite{BakerKegel} constructed an infinite family of hyperbolic L-space knots where the first knot is provably not braid positive, but the braid positivity status of the remaining knots remains unknown. We prove \Cref{LspaceIndex} for the Baker-Kegel family in \Cref{Discussion}.
\item Combining work of Hedden \cite{Hedden:CablingII} and Hom \cite{Hom:Cabling}, it is known that $K_{p,q}$ is an L-space knot if and only if $K$ is and $q \geq p(2g(K)-1)$. In \Cref{Discussion}, we show that if \Cref{LspaceIndex} is true for an L-space knot $K$, it is true for all of its L-space cables.
\end{itemize}

The Alexander polynomials of L-space knots have been of interest ever since Ozsv\'ath-Szab\'o proved that the coefficients of these polynomials take values in $\{-1,0,1\}$ \cite{OSz:LensSpaceSurgeries}. Hedden-Watson  \cite[Corollary 9]{HeddenWatson} later proved that when $K$ is an L-space knot, $\Delta_K(t) = 1-t+\ldots+ t^{2g-1} + t^{2g}$, where $g$ is the Seifert genus of $K$. The results we summarized above provide evidence for the following:

\begin{conj} \label{conj:AnotherTerm}
Suppose $K$ is a hyperbolic L-space knot with Seifert genus $g$ and braid index $n$. Then $\Delta_K(t) = 1-t+t^n + \ldots + t^{2g-n-1}- t^{2g-1}+ t^{2g}$. 
\end{conj}

Hyperbolicity is a key assumption in \Cref{conj:AnotherTerm}: the conjecture is false even for the (2,3)-cable of $T(2,3)$, which is an L-space knot. In a different direction, we also predict:

\begin{conj} \label{PositiveIndex}
If $K$ is a positive braid knot, the braid index and bridge index of $K$ agree.
\end{conj}

\Cref{PositiveIndex} is true for all positive braid knots in the knot tables \cite{KnotInfo}, but quickly becomes false for even slightly broader classes of knots. For example, the fibered, strongly quasipositive knot $10_{145}$ is not braid positive, and it has braid index is four while the bridge index is three \cite{KnotInfo}. It is worth noting that there are positive braid knots which do not admit a positive braid representative on a minimal number of strands \cite{Stoimenow:PositiveKnotsJonesPoly}, indicating a subtlety to \Cref{PositiveIndex}.

\subsection{Conventions}
\begin{itemize}[leftmargin=*]
\item[$\circ$] We let $B_n$ denote the braid group on $n$ strands. 
\item[$\circ$] Unless stated otherwise, we assume that all braid closures are knots. 
\item[$\circ$] When drawn vertically (resp. horizontally), our braids are oriented from north to south (resp. from west to east).
\item[$\circ$] We write $\beta_1 \approx \beta_2$ when braids are related by conjugation. We write $\beta_1 = \beta_2$ if they are related by braid relations. We underline the subword of a braid that is being modified. 
\item[$\circ$] Throughout, we denote the braid (resp. bridge) index of a knot $K$ by $i(K)$ (resp. $br(K)$). 
\item[$\circ$] We always work within the smooth category, and we only study knots up to smooth concordance. 
\end{itemize}

\subsection{Organization} We prove \Cref{thm:AlexPoly}, \Cref{thm:Lorenz}, and \Cref{thm:DistinctConcordanceClasses} in \Cref{AlexPolyResult}, \Cref{section:Lorenz}, and \Cref{section:Concordance}, respectively. \Cref{Discussion} provides evidence for \Cref{PositiveIndex}.

\subsection{Acknowledgements} We thank Joan Birman, Mike Miller Eismeier, Peter Feller, and Allison N. Miller for helpful conversations, and John Baldwin, Josh Greene, Eli Grigsby, Kyle Hayden, Matt Hedden, Francesco Lin, and Marissa Loving for comments on a preliminary version of this work. We used \textit{Kirby Calculator} (``KLO'') and \textit{KnotInfo} during the experimental phases of this work; we express our deep gratitude to Frank Swenton, Chuck Livingston, and Allison Moore for maintaining these invaluable resources. Finally, we thank the referee for their careful reading of this paper. SK was supported by NSF DMS-2103325.

\section{Twist Positivity and the Alexander Polynomial} \label{AlexPolyResult}

This section establishes a relationship between twist positivity and the Alexander polynomial. 

\noindent \textbf{\Cref{thm:AlexPoly}.} \textit{If $K$ is twist positive, then $\Delta_K(t) = 1 - t + t^n +t^{n+1}R(t)$, where $R(t) \in \Z[t]$.}

The proof of \Cref{thm:AlexPoly} requires the \textit{reduced Burau matrix}, $B(t)$, which is associated to a braid $\beta \in B_n$. We quickly recall the relevant details here, and recommend \cite{Birman:Book} for a full account. 

Let $\beta$ be a braid on $n$ strands. We write $\displaystyle \beta =\prod_{r=1}^m \sigma_{i_r}^{\epsilon_{i_r}}$, i.e. as the product of Artin generators.

To $\beta$, we can associate its reduced Burau matrix, $B(t)$, which is the product of $(n-1)\times (n-1)$ matrices $\sigma_i(t)$, where $\sigma_i(t)$ differs from the $(n-1)\times (n-1)$ identity matrix $I_{n-1}$ only in row $i$, as in \Cref{fig:Burau}. (Note that this matrix is truncated appropriately when $i=1,n-1$.) In particular, $\displaystyle B(t)=\prod_{r=1}^m \sigma_{i_r}^{\epsilon_{i_r}}(t)$.

\vspace{-1em}
\begin{figure}[h]
\centering
  $\sigma_i(t) =  \left( \begin{tabular}{ccccccc}
                    1&&&&&&\\ 
                    &$\ddots$&&&&&\\
                    &&1&&&&\\ 
                    &&$t$&$-t$&1&&\\ 
                    &&&&1&&\\ 
                    &&&&&$\ddots$&\\
                    &&&&&&1
                \end{tabular} \right)$
    \caption{The template for the matrices $\sigma_i(t)$, the building blocks of the matrix $B(t)$.}
    \label{fig:Burau}
\end{figure}

Suppose $\widehat{\beta}=K$, a knot. The Alexander polynomial for $K$, denoted $\Delta_K(t)$, is derived from $B(t)$ as follows: 
\[\frac{\det(I_{n-1}-B(t))}{1-t^n}=\frac{\Delta_K(t)}{1-t}\] 
In particular, $$\Delta_K(t)=(1-t)\cdot(\det(I_{n-1}-B(t)))\cdot(1+t^n+t^{2n}+\cdots).$$

We collect some important properties of $B(t)$:
\begin{itemize}
\item In general, the entries in a reduced Burau matrix lie in $\Z[t^{\pm 1}]$. However, when $\beta$ is a positive braid, the entries of $B(t)$ are genuine polynomial entries.
\item Recall: the full twist on $n$ strands is denoted by $\Delta^2$, where $\Delta^2 = (\sigma_1 \ldots \sigma_{n-1})^n \approx (\sigma_{n-1} \ldots \sigma_1)^n$. When $\beta = \Delta^2$, we have $B(t) = t^n I_{n-1}$. In particular, if $\beta$ is a twist positive braid, then $\beta = \Delta^2 \gamma$, so the reduced Burau matrix associated to $\beta$ is $B(t) = (t^n I_{n-1})(C(t)) = t^nC(t)$, where $C(t)$ is the reduced Burau matrix associated with $\gamma$.
\item Suppose $\beta$ is a positive braid on $n$ strands. We define the polynomial $q_{B(t)}(x)$ as follows:
\begin{align} 
q_{B(t)}(x):=\det(I_{n-1} - xB(t)) = 1+a_1(t)x+a_2(t)x^2 + \ldots + a_{n-1}(t)x^{n-1} \label{eqn:Burau}
\end{align}
We claim that the values of $a_1(t)$ and $a_{n-1}(t)$ can be quickly determined from $B(t)$: if we let $p_{B(t)}(x):=\det(xI_{n-1} - B(t))$ denote the characteristic polynomial of $B(t)$, then $q_{B(t)}(x)=x^{n-1}p_{B(t)}(x^{-1})$. That is, $q_{B(t)}(x)$ is obtained by $p_{B(t)}(x)$ by exchanging the coefficients of $x^i$ and $x^{(n-1)-i}$ for each $i \leq \lfloor (n-1)/2 \rfloor$. In particular, in $q_{B(t)}(t)$, $a_{n-1}(t) = (-1)^{n-1}\det(B(t))$ and $a_1(t)=-\tr(B(t))$.
\end{itemize}

We observe that substituting $x=t^n$ into \Cref{eqn:Burau} yields:
\begin{align}
\det(I_{n-1} - t^nB(t)) &= 1+a_1(t)t^n+a_2(t)t^{2n} + \ldots + t^{n^2-n}a_{n-1}(t) \nonumber \\
&= 1+a_1(t)t^n+t^{n+1} R(t) \label{eqn:PolyFullTwist}
\end{align}
where $R(t) \in \Z[t]$ collects all the higher order terms. 

With these preliminaries established, we can prove the result. 

\begin{proof}[Proof of \Cref{thm:AlexPoly}.]
Suppose $K = \widehat{\beta}$, where $\beta$ is a twist positive braid on $n$ strands. Then $\beta = \Delta^2 \gamma$, $B(t) = t^nC(t)$, and by applying the observation of \Cref{eqn:PolyFullTwist}, 
$$\det(I_{n-1} - B(t)) = \det(I_{n-1} - t^n C(t)) = 1+a_1(t)t^n+t^{n+1} R(t)$$ with $R(t) \in \Z[t]$. In particular, this means that
\begin{align}
\Delta_K(t)&=(1-t)\cdot\left(1+a_1(t)t^n+t^{n+1} R(t)\right)\cdot(1+t^n+t^{2n} + \ldots) \label{eqn:Setup} \\
&= 1-t+(1+a_1(t))t^n + t^{n+1}S(t) \label{eqn:AlexPoly}
\end{align}
where \Cref{eqn:AlexPoly} is obtained by distributing and rearranging the terms in the righthand side of \Cref{eqn:Setup}. Here, $S(t) \in \Z[t]$. 

\begin{claim} \label{Claim1}
Suppose $\widehat{\beta}$ is a knot $K$. Then the polynomial $a_1(t)$ has no constant term, i.e. $a_1(0)=0$. 
\end{claim}

\begin{proof}[Proof of \Cref{Claim1}.]
In the paragraph following \Cref{eqn:Burau}, we argued that $a_1(t) = -\tr(B(t))$. Therefore, to see that $a_1(0) = 0$, we must show that $\tr(B(0))=0$.

This is straightforward: first, we note that for all $i$, $\sigma_i(t)$ is an upper triangular matrix. Since the product of upper triangular matrices is also upper triangular, the matrix $B(t)$ is also an upper triangular matrix. Moreover, the diagonal entries of $B(t)$ are explicitly determined by the diagonal entries of its constituent matrices; more precisely,
\begin{align*}
\displaystyle [B(t)]_{ii}=\prod_{r=1}^m \left[\sigma_{i_r}^{\epsilon_{i_r}}(t)\right]_{ii}
\end{align*}
where $[ M ]_{ij}$ denotes the $(i,j)$ entry of the matrix $M$.

Since $\widehat{\beta}$ is a knot, then for all $1 \leq i \leq n-1$, $\sigma_i$ appears in $\beta$ at least once. Moreover, by the definition of the matrix $\sigma_i(t)$, $\left[\sigma_{i}(0)\right]_{ii} = 0$. Therefore, for all $1 \leq i \leq n-1$, $[B(t)]_{ii} = 0$. We deduce that $\tr(B(t)) = 0$, as desired.
\end{proof}

With this claim in hand, we can finish proving \Cref{thm:AlexPoly}. In particular, since $a_1(0)=0$, then $a_1(t) = e_1t + e_2t^2 + \ldots + e_kt^k$. Plugging this into \Cref{eqn:AlexPoly}, we have
\begin{align*}
\Delta_K(t) &= 1-t+(1+a_1(t))t^n + t^{n+1}S(t) \\
&= 1-t+(1 + (e_1t + e_2t + \ldots + e_kt^k))t^n + t^{n+1}S(t) \\
&= 1-t+t^n + t^{n+1}V(t)
\end{align*}
for some $V(t) \in \Z[t]$. This is exactly what we wanted to show.
\end{proof}

\begin{cor}
If $K$ is twist positive on $n$ strands, the braid index of $K$ appears as an exponent of the Alexander polynomial. 
\end{cor}

\begin{proof}
If $K$ is twist positive on $n$ strands, then by \cite{FranksWilliams}, the braid index of $K$ is $n$. On the other hand, \Cref{thm:AlexPoly} proves that $n$ appears as the exponent of the third term in $\Delta_K(t)$. 
\end{proof}

\section{Proof of \Cref{thm:Lorenz}} \label{section:Lorenz}

First, we establish that some well-studied families of knots are twist positive:

\begin{prop} \label{ExamplesOfLorenzKnots}
Algebraic knots, 1-bridge braids, and many twisted torus knots are twist positive L-space knots. In fact, something stronger is true: these knots are all L-space Lorenz knots.
\end{prop}
 
\begin{proof}
Birman-Williams proved that every Lorenz knot is twist positive \cite{BirmanWilliams}. Moreover, Birman-Kofman \cite{BirmanKofman} showed that Lorenz knots are in correspondence with \textit{T-links}, which are defined to be the closures of positive braids of a particular form:
\begin{align} \label{TLink}
\tau = (\sigma_1 \sigma_2 \ldots \sigma_{p_1-1})^{q_1} (\sigma_1 \sigma_2 \ldots \sigma_{p_2-1})^{q_2} \ldots (\sigma_1 \sigma_2 \ldots \sigma_{p_s-1})^{q_s}
\end{align}

where $2 \leq p_1 \leq p_2 \leq \ldots \leq p_s$, $0 < q_i$ for all $i$, and $\tau$ is a braid in $B_{p_s}$. Thus, to prove \Cref{ExamplesOfLorenzKnots}, it suffices to show that algebraic knots, 1-bridge braids, and twisted torus knots are $T$-links, and also confirm their L-space status. 

\begin{itemize}
\item Birman-Williams showed that algebraic knots are Lorenz knots \cite[Theorem 6.3]{BirmanWilliams}. It is well known that algebraic knots are particular iterated cables of torus knots \cite{Brauner}; applying Hedden's results on cabling \cite{Hedden:CablingII}, we deduce they are L-space knots.
\item A \textit{1-bridge braid} is a knot $K \subset S^3$ realized as the closure of a braid $\beta$, where:
$$\beta = (\sigma_b \sigma_{b-1}\ldots \sigma_2 \sigma_1) (\sigma_{w-1}\sigma_{w-2}\ldots \sigma_2 \sigma_1)^t$$
\cite{REU2021} proved that all 1-bridge braids admit $T$-link presentations (hence they are twist positive), and \cite{GLV:11LSpace} proved they are L-space knots.
\item In \cite{Vafaee:TwistedTorusKnots}, Vafaee presents twisted torus knots as the closures of positive braids on $w$ strands with the following form: $\rho = (\sigma_{w-1} \sigma_{w-2} \ldots \sigma_1)^t (\sigma_{w-1} \sigma_{w-2} \ldots \sigma_{w-k})^{sk}$, i.e. they are built from a torus link by inserting some full twists into $k$ adjacent strands. \cite{REU2021} showed that all twisted torus knots admit $T$-link presentations, so they are twist positive;  \cite{Vafaee:TwistedTorusKnots} proved that for certain parameters, they are L-space knots.
\end{itemize}
\end{proof}

\subsection{Twist positive L-space knots} We begin by proving \Cref{thm:Lorenz}.  

\vspace{1em}

\textbf{\Cref{thm:Lorenz}.} \textit{
If $K$ is a twist positive L-space knot, then its braid index and bridge index agree. Moreover, $K$ realizes the minimal braid and bridge indices of its concordance class.} 

\begin{proof}
If $K$ is a twist positive L-space knot, then by \Cref{thm:AlexPoly}, we have
$$\Delta_K(t) = 1 - t + t^n + t^{n+1}R(t)$$ By \cite{FranksWilliams}, we know the braid index of $K$ is $n$. Moreover, we see that $n-1$ appears as a gap in the exponents of $\Delta_K(t)$. Juh\'asz-Miller-Zemke \cite{JuhaszMillerZemke} established a relationship between knot Floer homology and the bridge index, which we now describe. We recall: the \textit{minus version} of knot Floer homology, denoted $HFK^-(K)$, is a knot invariant, and has the structure of a finitely generated module over $\mathbb{F}_2[v]$. Moreover, there is a (non-canonical) decomposition of $HFK^{-}(K)$ as follows: $ HFK^{-}(K) \cong \mathbb{F}_2[v] \oplus HFK^{-}_{red}(K)$, where $HFK^{-}_{red}(K)$ is the $\mathbb{F}_2[v]$-torsion submodule of $HFK^{-}(K)$. Juh\'asz-Miller-Zemke defined the knot invariant $Ord_v(K)$ (i.e the \textit{torsion order} of $K$) to be the minimal integer which annihilates this torsion submodule; for more details, we refer the reader to \cite[Section 3]{JuhaszMillerZemke}.

\begin{thm} \cite[Corollary 1.9]{JuhaszMillerZemke} \label{TorsionOrder}
If $K$ is a knot in $S^3$, $Ord_v(K) \leq br(K)-1$ $($where $Ord_v(K)$ denotes the torsion order of $K$, an invariant from knot Floer homology$)$. \hfill $\Box$
\end{thm}

Let $\alpha_0, \ldots, \alpha_{2m}$ be the non-zero exponents appearing in $\Delta_K(t)$, written in decreasing order. Juh\'asz-Miller-Zemke showed that for L-space knots, $Ord_v(K) = \max \{\alpha_{i-1} - \alpha_i | 1 \leq i \leq 2m \}$  \cite[Lemma 5.1]{JuhaszMillerZemke}. Since $n-1$ appears as a gap in the exponents of $\Delta_K(t)$, we have $n-1 \leq Ord_v(K)$. The bridge index, $br(K)$, is a lower bound on the braid index, $i(K)$, so:
\begin{align}
n-1 &\leq Ord_v(K) \leq br(K)-1 \leq i(K)-1 = n-1 \label{eqn:Inequalities}
\end{align}

But this means we must have equalities throughout, hence $br(K) = n = i(K)$, as desired.

Next, we prove that twist positive L-space knots realize the minimal braid index of their concordance classes.

The proof of this is nearly identical to the proof of \cite[Corollary 1.10]{JuhaszMillerZemke}. Suppose the knot $J$ is concordant to a twist positive L-space knot $K$, and let $N(K)$ denote the concordance invariant defined by Dai-Hom-Truong-Stoffregen \cite{DHST:ConcordanceHomomorphisms}. We have the following string of inequalities (justification for the respective (in)equalities are in the paragraph following):
\begin{align}
i(J) \geq br(J) &\geq Ord_v(J)+1 \\
& \geq N(J) + 1 \\
&= N(K) + 1 \\
&= Ord_v(K) + 1 \\
&= br(K) = i(K)
\end{align}

The first equality of (3.3) is standard; the second is a restatement of \Cref{TorsionOrder}. Dai-Hom-Truong-Stoffregen \cite{DHST:ConcordanceHomomorphisms} showed that $N(K)$ is a lower bound on $Ord_v(K)$, yielding the inequality in (3.4). Since $N(K)$ is a concordance invariant, we have the equality in (3.5). \cite{JuhaszMillerZemke} showed that for L-space knots, the invariants $N(\star)$ and $Ord_v(\star)$ agree; this yields the equality in (3.6). The equalities of (3.7) follow from those in (\ref{eqn:Inequalities}).
\end{proof}

\section{Proof of \Cref{thm:DistinctConcordanceClasses}} \label{section:Concordance}

In this section, we prove that L-space knots of braid index three lie in distinct concordance classes, and these classes are distinct from those of torus knots. 

\subsection{Preliminaries} We begin with some preliminaries.

\begin{defn} \label{defn:TTK}
$T(3, k; 2m)$ denotes the twisted torus knot which is the closure of $(\sigma_2 \sigma_1)^k (\sigma_2)^{2m}$, where $k \geq 4$, $k \not \equiv 0 \mod 3$, and $m \geq 0$. 
\end{defn}

We briefly justify why we may assume that $k$ is at least 4: if $k=1, 2$, then $T(3,k; 2m)$ is isotopic to a $T(2,r)$ torus knot. Additionally, for all $m \geq 0$ and $k \equiv 0 \mod 3$, $T(3,k;2m)$ is a link. Thus, we may assume $k \geq 4$. We note that when $k \geq 4$, $T(3,k; 2m)$ is twist positive on three strands.

\begin{lemma} \label{lemma:hyperbolic}
Every L-space knot of braid index three is either a torus knot or a hyperbolic knot.
\end{lemma}

\begin{proof}
Thurston proved that every knot is either a torus knot, a satellite knot, or hyperbolic. Suppose $K$ is an L-space knot of braid index three. By Vafaee \cite[Corollary 3.3]{Vafaee:TwistedTorusKnots} and Lee-Vafaee \cite{LeeVafaee}, $K$ is a twisted torus knot on three strands. 

Suppose $K$ is not a torus knot. To show that it is hyperbolic, we must show that it is not a satellite knot, i.e. $K$ does not contain an essential torus in its exterior. Birman and Menasco \cite{BirmanMenasco:EssentialTori} proved that if $L$ is a link in $S^3$ of braid index three, and $L$ contains an essential, non-peripheral torus in its exterior, then $L$ is the closure of a braid $\beta$ which is conjugate to $\gamma_{p,q} = (\sigma_2)^p (\sigma_1 \sigma_2 \sigma_2 \sigma_1)^q$, where $|p| \geq 2$ and $|q| \geq 1$. These braid closures are always links: the cycle in the symmetric group $S_3$ induced by $\gamma_{p,q}$ contains a 1-cycle. We deduce the non-torus twisted torus knots are hyperbolic. 
\end{proof}

In the next lemma, we see that it suffices to study $T(3,k;2m)$ where $k \equiv 1 \mod 3$.

\begin{lemma} \label{lemma:AdjacentTTKs}
Suppose $q \equiv 1 \mod 3$ and $m \geq 1$. Then the twisted torus knots $T(3,q; 2m)$ and $T(3,q+1; 2m-2)$ are isotopic.
\end{lemma}

\begin{proof}
We show that the braid words for $T(3,q+1; 2m-2)$ and $T(3,q;2m)$ are related by braid relations and conjugation. As $q \equiv 1 \mod 3$, we will write $q = 3r + 1$. 

First, recall that conjugating a braid $\beta$ in $B_n$ by the Garside element $\Delta = (\sigma_1)(\sigma_1 \sigma_2) \ldots (\sigma_1 \sigma_2 \ldots \sigma_{n-1})$ has the effect of exchanging $\sigma_i$ and $\sigma_{n-i}$ for all $1 \leq i \leq n-1$. In particular, when $n=3$, conjugating a 3-braid by the Garside element corresponds to swapping $\sigma_1$ and $\sigma_2$. This conjugation procedure justifies the equivalence in the first line of the sequence of braid moves below.

\begin{align*}
(\sigma_2 \sigma_1)^{q} (\sigma_2)^{2m} &\approx (\sigma_1 \sigma_2)^{q} (\sigma_1)^{2m} \\
&= (\sigma_1 \sigma_2)^{3r} (\sigma_1 \sigma_2) \sigma_1 (\sigma_1)^{2m-1}\\
&= \underline{(\sigma_1 \sigma_2)^{3r}} \ (\underline{\sigma_1 \sigma_2 \sigma_1}) (\sigma_1)^{2m-1}
\intertext{Since $(\sigma_1 \sigma_2)^3 = (\sigma_2 \sigma_1)^3$, we have:}
%
%
&= (\sigma_2 \sigma_1)^{3r} (\sigma_2 \sigma_1 \sigma_2) (\sigma_1)^{2m-1}\\
&= (\sigma_2 \sigma_1)^{3r} (\sigma_2 \sigma_1 \sigma_2) \sigma_1 (\sigma_1)^{2m-2}\\
&= (\sigma_2 \sigma_1)^{3r+2} (\sigma_1)^{2m-2}\\
&\approx (\sigma_1)^{2m-2} (\sigma_2 \sigma_1)^{3r+2} \\
&= (\sigma_1)^{2m-3} \underline{(\sigma_1) (\sigma_2 \sigma_1)} (\sigma_2 \sigma_1)^{3r+1} \\
&= (\sigma_1)^{2m-3} (\sigma_2 \sigma_1 \sigma_2) (\sigma_2 \sigma_1)^{3r+1} \\
&= (\sigma_1)^{2m-4} (\underline{\sigma_1 \sigma_2 \sigma_1}) \sigma_2 (\sigma_2 \sigma_1)^{3r+1} \\
&= (\sigma_1)^{2m-4} (\sigma_2 \sigma_1 \sigma_2) \sigma_2 (\sigma_2 \sigma_1)^{3r+1}\\
&= (\sigma_1)^{2m-4} (\sigma_2 \sigma_1) \sigma_2^2 (\sigma_2 \sigma_1)^{3r+1}
\intertext{We continue extracting a $\sigma_1$ term from the leftmost parenthetical, and then applying a braid relation to the leftmost subword $\sigma_1 \sigma_2 \sigma_1$. At the penultimate application of this process, we get:}
&= (\sigma_1)(\sigma_2 \sigma_1\sigma_2)\sigma_2^{2m-4} (\sigma_2 \sigma_1)^{3r+1} \\
&= (\underline{\sigma_1\sigma_2 \sigma_1})\sigma_2^{2m-3} (\sigma_2 \sigma_1)^{3r+1} \\
&= (\sigma_2\sigma_1 \sigma_2)\sigma_2^{2m-3} (\sigma_2 \sigma_1)^{3r+1} \\
%
%
&\approx  (\sigma_2 \sigma_1)^{3r+2}\sigma_2^{2m-2} 
\end{align*}
But this is exactly the braid whose closure is $T(3,q+1; 2(m-1))$. Thus, when $q \equiv 1 \mod 3$ and $m \geq 1$, the twisted torus knots $T(3, q; 2m)$ and $T(3, q+1; 2(m-1))$ are isotopic. 
\end{proof}

\begin{cor} \label{StandardForm}
Every positive twisted torus knot on three strands can be presented as a positive twisted torus knot $T(3,q;2m)$ where $q \geq 4$, $q \equiv 1 \mod 3$ and $m \geq 0$. \hfill $\Box$
\end{cor}

\begin{prop} \label{lemma:PotentialConcordance}
Fix a torus knot $K = T(3,q) = T(3,q;0)$ with $q \geq 4$. There are $\lfloor \frac{q}{3}\rfloor-1$ other twisted torus knots $T$ of braid index 3 with $\tau(T)=\tau(K)$ where $\tau$ denotes the Ozsv\'ath-Szab\'o concordance invariant.
\end{prop}

In fact, our proof of \Cref{thm:DistinctConcordanceClasses} will show that the $\lfloor \frac{q}{3}\rfloor-1$ twisted torus knots are pairwise distinct from $T(3,q)$ and from each other. 

\begin{proof}
Since $T$ and $K$ are both positive braid knots, they are fibered and strongly quasipositive \cite{Stallings:Fibered}. By Hedden, $g_4(T) = g(T) = \tau(T)$ and $g_4(K)=g(K)=\tau(K)$ \cite{Hedden:Positivity}. To deduce the statement, it suffices to determine which twisted torus knots have the same Seifert genus as $K$. 

For positive braid knots, the standard Bennequin surface is a fiber surface for the knot \cite{Stallings:Fibered}. It is therefore straightforward to compute the Seifert genus of any positive braid knot $P = \widehat{\rho}$, where $\rho$ is a positive braid on $n$ strands: $2g(P)-1 = wr(\rho)-n$. If two positive $n$-braids have the same writhe, then they have the same Seifert genus, hence the same $\tau$. Therefore, our problem reduces to counting the number of twisted torus knots on 3 strands with the same writhe as $K$.

First, suppose that $q \equiv 1 \mod 3$, so $q = 3r+1$, where $r \geq 1$. The writhe of $T(3,q) = T(3,q; 0)$ is $2q = 2(3r+1)$. Note that a twisted torus knot of the form $T(3, q+\epsilon; 2m)$ with $\epsilon \geq 1, m \geq 0$ can never be concordant to $T(3,q)$, as these knots do not have the same writhe. However, it is elementary to check that for all $s$ with $0 \leq s \leq r-1$, the knots $T(3, 3(r-s)+1; 6s)$ all have writhe $6r+2$; thus, they have the same Seifert genus. By \Cref{lemma:AdjacentTTKs}, the twisted torus knots $T(3, 3(r-s)+1; 6s)$ and $T(3, 3(r-s)+2; 6s-2)$ are isotopic knots. Therefore, to determine the number of positive twisted torus knots $T$ with $\tau(T)=\tau(K)$, it suffices to count how many twisted torus knots of the form $T(3, 3(r-s)+1; 6s)$ there are, where $r-1 \geq s \geq 1$. A simple counting argument allows us to deduce that there are
$\frac{q-4}{3} = \lfloor \frac{q}{3} \rfloor -1$
such knots.

Next, suppose $q \equiv 2 \mod 3$, so we are studying torus knots of the form $T(3,q) = T(3, 3r+2)$. By \Cref{lemma:AdjacentTTKs}, the knots $T(3,3r+2)$ and $T(3, 3r+1; 2)$ are isotopic. Applying the same argument as in the previous paragraph, every twisted torus knot of the form $T(3, 3(r-s)+1; 2+6s)$, where $r-1 \geq s \geq 1$, will have the same writhe as $T(3, 3r+1; 2)$. There are $\lfloor \frac{q}{3} \rfloor -1$ such knots.

Therefore, there are $\lfloor \frac{q}{3} \rfloor -1$ many twisted torus knots $T$ with $\tau(T)=\tau(K)$, regardless of the congruence class of $q \mod 3$.
\end{proof}

Our long-term goal is to show that, despite there being many twisted torus knots with the same $\tau$ invariants, these knots represent distinct concordance classes. This is proved by computing the signatures of 3-stranded twisted torus knots.

\subsection{Computing the signature via the Goeritz matrix}

To compute the signature, we use the Gordon-Litherland \cite{GordonLitherland:Signature} method, which now we recall: let $\mathcal{K}$ be a knot in $S^3$, and fix a regular projection $K$ of $\mathcal{K}$. First, we checkerboard color $K$, where the checkerboard has black and white regions (without loss of generality, we assume that the unbounded region is white). Each white region of the diagram is labelled $X_i$, and we set $X_0$ to be the unbounded white region. To each double point $D$ in $K$, we assign a value $\eta(D) = \pm 1$, as in \Cref{fig:Goeritz}. Each double point $D$ is also assigned a \textit{type}, i.e. it is either \textit{Type I} or \textit{Type II}; see \Cref{fig:Goeritz}.

This information is organized into a symmetric, integral matrix $G'(K)$, with entries $g_{ij}$, where:
\begin{align*}
g_{ij} = \begin{cases}
 - \displaystyle \sum \eta(D) & \text{summed over double points $D$ incident to $X_i$ and $X_j$, where $i \neq j$} \\ \\
\displaystyle- \sum_{k, k \neq i} g_{ik} & \text{if $i = j$}
\end{cases}
\end{align*} 

The \textit{Goeritz matrix} $G(K)$ is constructed from $G'(K)$ by deleting the $0^{th}$ row and column of $G'(K)$. Thus, $G$ is a symmetric, integral, $n \times n$ matrix, where $n$ is the number of bounded white regions in $K$. Gordon-Litherland \cite{GordonLitherland:Signature} proved that 
\begin{align} \label{eqn:GordonLitherland}
\sigma(\mathcal{K}) = \text{sign}(G(K))-\mu(K)
\end{align}

where $\mu(K) = \sum \eta(D)$, where $D$ is a Type II double point in $K$. 

\vspace{-1cm}
\begin{figure}[H]
\labellist \tiny
\pinlabel {$\eta(D) = +1$} at 120 -30 
\pinlabel {$\eta(D) = -1$} at 440 -30 
\pinlabel {Type I} at 1050 -30 
\pinlabel {Type II} at 1380 -30 
\endlabellist
\begin{center}
\vspace{1cm}
\hspace{-2em}
\includegraphics[scale=.25]{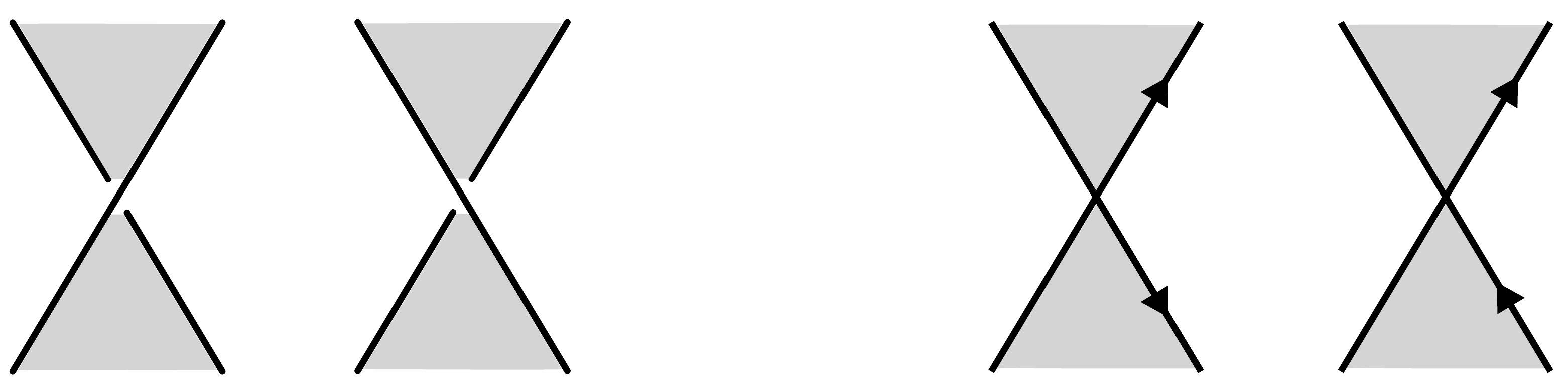}
\end{center}
\vspace*{2mm}
\captionof{figure}{Conventions for $\eta(D)$ are on the \textbf{left}, and assignments for \textit{Type I} and \textit{II} double points are on the \textbf{right}.}
\label{fig:Goeritz}
\end{figure}

We are now ready to compute the signature of $T(3,q;2m)$ where $q \geq 4$ and $m \geq 0$.

\begin{prop} \label{GoeritzSetup}
Let $K = T(3,q; 2m)$ denote a twisted torus knot with $q=3k+1 \equiv 1 \mod 3$, where $k \geq 1$ and $m \geq 0$. There is a $q \times q$ Goeritz matrix $G(K)$ for $K$, where all the non-zero entries are either along the tridiagonal, or in the entries $(1,q)$ or $(q,1)$. 
\end{prop}

\begin{proof} 
Throughout, we conflate a knot and a regular projection of it.

Let $K$ be the specified knot. Rather than presenting $K$ as the closure of $(\sigma_2 \sigma_1)^{q}(\sigma_2)^{2m}$, we conjugate the braid by the Garside element (as in \Cref{lemma:AdjacentTTKs}), and view $K$ as the closure of $(\sigma_1 \sigma_2)^q (\sigma_1)^{2m}$. Next, we modify the braid:
\begin{align*}
(\sigma_1 \sigma_2)^q (\sigma_1)^{2m} &=(\underline{\sigma_1} \sigma_2) (\sigma_1 \sigma_2)^{q-2} (\sigma_1 \sigma_2) (\sigma_1) (\sigma_1)^{2m-1} \\
&\approx (\sigma_2) (\sigma_1 \sigma_2)^{q-2} (\sigma_1 \sigma_2) (\sigma_1)(\sigma_1)^{2m-1} (\sigma_1)\\
&=(\sigma_2 \sigma_1)^{q}(\sigma_1)^{2m}\\
&=: \beta
\end{align*}

This presentation is well adapted to find a spanning surface. Indeed, we can produce a checkerboard coloring with exactly $q+1$ white regions: present $K$ as the closure of $\beta$, and read $\beta$ from left to right. For $i \neq 0, q$, the white region $X_i$ lies between the $i^{th}$ and $(i+1)^{th}$ occurrences of $\sigma_2$ in $\beta$. $X_{q}$ is the white region whose boundary contains the following double points: the first $\sigma_2$ in $\beta$, the last $\sigma_2$ in $\beta$, and double points from $(\sigma_1)^{2m}$. See \Cref{fig:Example} for this labelling demonstrated on $K = T(3,7;4)$.

\begin{figure}[H]
\begin{tikzpicture}
    \draw (1, 0) node[inner sep=0] {\includegraphics[scale=0.7]{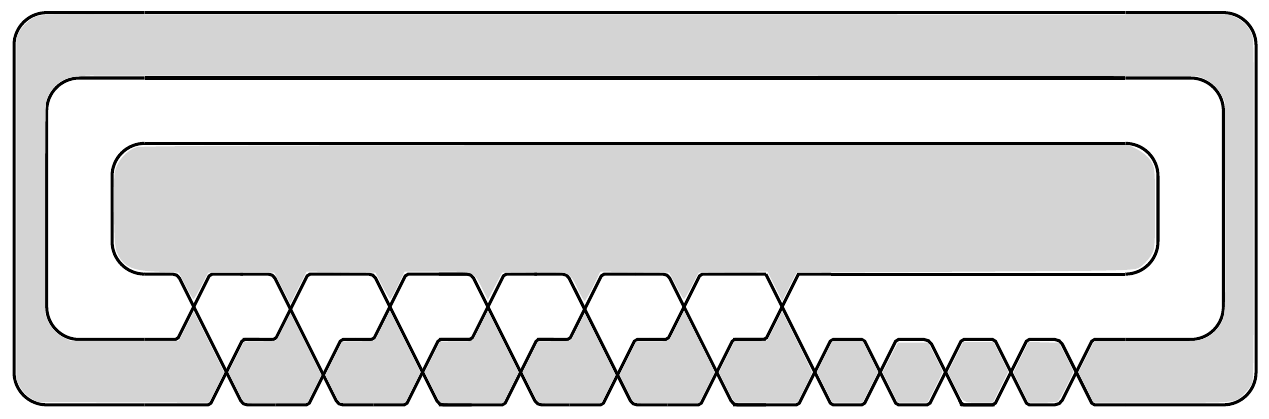}};
    \draw (-3.75, -1.2) node {$X_1$};
    \draw (-2.5, -1.2) node {$X_2$};
    \draw (-1.3, -1.2) node {$X_3$};
    \draw (-0.2, -1.2) node {$X_4$};
    \draw (1, -1.2) node {$X_5$};
    \draw (2.2, -1.2) node {$X_6$};
    \draw (5, -1.2) node {$X_7$};
\end{tikzpicture}
\captionof{figure}{The twisted torus knot $T(3,7;4)$ and a well-adapted spanning surface used to identify a Goeritz matrix. The label for the unbounded region, $X_0$, is suppressed. In this figure, $\sigma_1$ is closer to the bottom of the page, and $\sigma_2$ is closer to the top of the page. }
\label{fig:Example}
\end{figure}

We now study the double points/crossings, and understand $\eta(D)$ and the type of each crossing. By inspection, we see that each double point coming from a $\sigma_1$ has $\eta(D) = +1$, while each double point coming from a $\sigma_2$ has $\eta(D) = -1$. Additionally, each $\sigma_1$ corresponds to a Type II double point, while each $\sigma_2$ corresponds to a Type I double point. In particular, for $T(3, q; 2m)$, we have:
\begin{align} \label{GordonLitherland_Mu}
\mu(K)&=\sum_{\substack{\text{Type II}\\ \text{double points}}}\eta(D) = \sum_{\sigma_1 \in  \beta} (+1)  = q + 2m
\end{align}

We now construct the $(q+1) \times (q+1)$ matrix $G'(K)$. We begin by collecting some observations about the incidence of the white regions, and what that tells us about the non-diagonal entries $g_{ij}$. For clarity of exposition, we let $D_i$ denote a double point arising from $\sigma_i$. As we observed above, $\eta(D_1)=1$ and $\eta(D_2)=-1$.

\begin{itemize}
\item When $i \neq q$, the white regions $X_0$ and $X_i$ meet in exactly one double point corresponding to the $i^{th}$ occurrence of $\sigma_1$ in $\beta$. Therefore, for $i \not \in \{0,q\}$, $g_{0,i} = g_{i,0}= -\eta(D_1)=-1$.
\item The white regions $X_0$ and $X_q$ meet in $2m+1$ double points; each double point arises from a single $\sigma_1$ which has $\eta(D_1)=1$. Thus, $g_{0,q} = g_{q,0}=-(2m+1)$. 
\item For distinct, non-zero $i$ and $j$, the regions $X_i$ and $X_j$ meet in exactly one double point (coming from a single occurrence of $\sigma_2$) if and only if $|i-j| \equiv \pm 1 \mod q$. Therefore, in this case, $g_{ij}=g_{ji}=-\eta(D_2)=+1$.
\end{itemize}

This is enough information to tell us about the diagonal entries of $G'(K)$:
\begin{itemize}
\item $g_{0,0} = - (g_{1,0} + g_{2,0} + \ldots + g_{q-1,0} + g_{q,0}) = -\left((q-1)(-1) + (-1)(2m+1)\right) = 2m+q$.
\item For $i \not \in \{0, q\}$, $g_{i,i} = -(g_{0,i} + g_{i-1,i} + g_{i+1,i}) = -(-1 + 1 + 1) = -1$.
\item $g_{q,q} = -(g_{0,q} + g_{1,q} + g_{q-1,q}) = -((-1)(2m+1) + 1 + 1) = 2m-1$. 
\end{itemize}

Since $G(K)$ is obtained by deleting the zero-th row and column of $G'(K)$, the matrix $G_k$ is the $(3k+1) \times (3k+1)$ matrix with the following entries:
\begin{align*}
g_{ij} = \begin{cases}
-1 & i=j \neq q \\
2m-1 & i=j=q \\
+1 & |i-j| \equiv \pm 1 \mod q \\
0 & \text{otherwise}
\end{cases}
\end{align*}

As claimed, $G(K)$ has non-zero entries in the tridiagonal, and in the entries $(1,q)$ and $(q,1)$. 
\end{proof}

The matrices $G'(K)$ and $G(K)$ for the twisted torus knot $T(3,7;6)$ from \Cref{fig:Example} are below.
\[ 
\hspace{-0.5cm}
G'(K) =  
\begin{blockarray}{ccccccccc}
& \scriptscriptstyle{\textcolor{gray}{0}} & \scriptscriptstyle{\textcolor{gray}{1}} & \scriptscriptstyle{\textcolor{gray}{2}} & \scriptscriptstyle{\textcolor{gray}{3}} & \scriptscriptstyle{\textcolor{gray}{4}} & \scriptscriptstyle{\textcolor{gray}{5}} & \scriptscriptstyle{\textcolor{gray}{6}} & \scriptscriptstyle{\textcolor{gray}{7}} \\
\begin{block}{c(cccccccc)}
\scriptscriptstyle{\textcolor{gray}{0}}&11 & -1 & -1 & -1 & -1 & -1 & -1 & -5 \\
\scriptscriptstyle{\textcolor{gray}{1}}&-1 & -1 & +1 & 0 & 0 & 0 & 0 & +1 \\
\scriptscriptstyle{\textcolor{gray}{2}}& -1 & +1 & -1 & +1 & 0 & 0 & 0 & 0 \\
\scriptscriptstyle{\textcolor{gray}{3}}& -1 & 0 & +1 & -1 & +1 & 0 & 0 & 0 \\
\scriptscriptstyle{\textcolor{gray}{4}}& -1 & 0 & 0 & +1 & -1 & +1 & 0  & 0\\
\scriptscriptstyle{\textcolor{gray}{5}}& -1 & 0 & 0 & 0 & +1 & -1 & +1 & 0 \\
\scriptscriptstyle{\textcolor{gray}{6}}&-1 & 0 & 0 & 0 & 0 & +1 & -1 & +1 \\
\scriptscriptstyle{\textcolor{gray}{7}}&-5 & +1 & 0 & 0 & 0 & 0 & +1 & +3 \\
\end{block}
\end{blockarray}
\qquad G(K) = 
 \begin{pmatrix}
    \begin{array}{ccccccc}
 -1 & +1 & 0 & 0 & 0 & 0 & +1 \\
 +1 & -1 & +1 & 0 & 0 & 0 & 0 \\
 0 & +1 & -1 & +1 & 0 & 0 & 0 \\
 0 & 0 & +1 & -1 & +1 & 0  & 0\\
 0 & 0 & 0 & +1 & -1 & +1 & 0 \\
 0 & 0 & 0 & 0 & +1 & -1 & +1 \\
 +1 & 0 & 0 & 0 & 0 & +1 & +3 \\
\end{array}
\end{pmatrix}
\]

Our goal is to compute the signature of $G(K)$. We will need a powerful theorem, typically referred to as ``Sylvester's Law of Inertia'': \textit{two symmetric, $n \times n$ matrices have the same number of positive, negative, and zero eigenvalues if and only if they are congruent} \cite[Theorem 9.13]{Carrell:LinearAlgebraBook}. (Recall that two matrices $A$ and $B$ are congruent if there exists a non-singular matrix $S$ such that $B = SAS^T$.) We will prove that $G(K)$ is congruent to a diagonal matrix $D$. Since, for diagonal matrices, the eigenvalues are exactly the diagonal entries, we can compute the signature of $G(K)$ from $D$. 

\begin{prop} \label{DiagonalizeG(K)} Let $K$ be the positive twisted torus knot $T(3, 3k+1;2m)$, where $k \geq 1, m \geq 0$. Then the Goeritz matrix $G(K)$ is congruent to the diagonal matrix $D$, where 
\begin{align*}
d_{ii} = \begin{cases}
(2m-1) \pm 2 & i = q \\
1 & i \equiv 0 \mod 3 \\
-1 & \text{otherwise}
\end{cases}
\end{align*}
\end{prop} 

To prove this result, we will first establish some preliminaries.

\begin{defn}
The symmetric $(3 \ell +1) \times (3 \ell+1)$ matrices $N_{\ell, \epsilon} = (n_{ij})$ and $P_{\ell, \epsilon} = (p_{ij})$ have entries as defined in \Cref{fig:DefiningPN} (left). 
\end{defn}

We note: the first parameter in the subscript identifies the dimension of the matrix, while the second index records the value $\epsilon$. \Cref{fig:DefiningPN} (right) displays $N_{1,\epsilon}$ and $P_{1, \epsilon}$. We note that the entries of $P_{\ell, \epsilon}$ and $N_{\ell, \epsilon}$ agree in all but two entries: $n_{1, 3\ell+1} = n_{3\ell+1,1}=-1$ while $p_{1, 3\ell+1} = p_{3\ell+1,1}=+1$ (indeed, $N$ and $P$ stand for ``negative'' and ``positive'', respectively).

\begin{figure}[h!]
    \begin{align*}
    n_{ij} &= \begin{cases}
    -1 & i=j \neq 3 \ell+1, \text{ or } \{i, j\} = \{1, q\} \\
    \epsilon & i=j=3 \ell+1 \\
    +1 & |i-j| = \pm 1, \text{ and either } i \geq 2 \text{ or } j \geq 2  \\
    0 & \text{otherwise}
    \end{cases} 
    \hspace{1.5cm}
    N_{1,\epsilon} = 
    \begin{pmatrix}
    \begin{array}{ccc|c}
     -1 & +1 & 0 & -1 \\
    +1 & -1 & +1 & 0 \\
     0 & +1 & -1 & +1 \\ \hline
    -1 & 0 & +1 & \epsilon \\
    \end{array}
    \end{pmatrix} 
    \\ \\
    p_{ij} &= \begin{cases}
    -1 & i=j \neq q \\
    \epsilon & i=j=q \\
    +1 & |i-j| \equiv \pm 1 \mod q \\
    0 & \text{otherwise}
    \end{cases} 
    \hspace{4.5cm}
    P_{1,\epsilon} = 
    \begin{pmatrix}
    \begin{array}{ccc|c}
     -1 & +1 & 0 & +1 \\
    +1 & -1 & +1 & 0 \\
     0 & +1 & -1 & +1 \\ \hline
    +1 & 0 & +1 & \epsilon \\
    \end{array}
    \end{pmatrix} 
    \end{align*}
\caption{Defining the matrices $N_{\ell, \epsilon}$ and $P_{\ell, \epsilon}$.}
\label{fig:DefiningPN}
\end{figure}

\begin{defn}
The matrix $B$ is the $3 \times 3$ diagonal matrix with $d_{1,1} = d_{2,2}=-1$ and $d_{3,3}=+1$.
\end{defn}

\begin{lemma} \label{lemma:BigLemma}
For all $q \geq 1$, $P_{q, \epsilon}$ is congruent to the block diagonal matrix with $q$-many $B$ factors, and a single $1 \times 1$ factor with entry $\epsilon \pm 2$. 
\end{lemma}

\begin{proof}

We will need (strong) induction to prove the claim. To begin, we prove:

\begin{claim} \label{claim:PreBaseCase_N}
$N_{1,\epsilon}$ and $P_{1,\epsilon}$ are congruent to $4 \times 4$ block diagonal matrices containing $B$ as a factor, as in \Cref{fig:P1N1}.
\end{claim}

\begin{figure}[h!]
\begin{align*}
N_{1,\epsilon} &\sim 
\begin{pmatrix}
\begin{array}{ccc|c}
 -1 & 0 & 0 & 0 \\
0 & -1 & 0 & 0 \\
 0 & 0 & +1 & 0 \\ \hline
0 & 0 & 0 & \epsilon+2 \\
\end{array}
\end{pmatrix} 
\hspace{1.5cm}
P_{1,\epsilon} \sim 
\begin{pmatrix}
\begin{array}{ccc|c}
 -1 & 0 & 0 & 0 \\
0 & -1 & 0 & 0 \\
 0 & 0 & +1 & 0 \\ \hline
0 & 0 & 0 & \epsilon-2 \\
\end{array}\end{pmatrix} 
\end{align*}
\caption{The matrices $N_{1,\epsilon}$ and $P_{1,\epsilon}$ are congruent to block diagonal matrices.}
\label{fig:P1N1}
\end{figure}

\textit{Proof of \Cref{claim:PreBaseCase_N}:} 
We begin by showing $N_{1,\epsilon}$ can be diagonalized into the claimed form. The sequence of simultaneous row and column operations is written out explicitly in \Cref{fig:MatrixBaseCaseExplained}. In Step 1 and 2, we simultaneously clear the non-zero, non-diagonal entries in the first row/column. In Step 3, we simultaneously swap the second and third rows and columns. In Steps 4 and 5, we clear the non-zero, non-diagonal entries of the second row/column.

The analogous sequence of operations used to diagonalize $P_{1,\epsilon}$ are in \Cref{fig:PConguentD}. 
\hfill $\Box$

\begin{figure}[h!]
\begin{align*}
N_{1,\epsilon} =
\begin{pmatrix}
\begin{array}{ccc|c}
 -1 & +1 & 0 & -1 \\
+1 & -1 & +1 & 0 \\
 0 & +1 & -1 & +1 \\ \hline
-1 & 0 & +1 & \epsilon \\
\end{array}
\end{pmatrix} 
&\overset{1}{\sim}
\begin{pmatrix}
\begin{array}{ccc|c}
 -1 & 0 & 0 & -1 \\
0 & 0 & +1 & -1 \\
 0 & +1 & -1 & +1 \\ \hline
-1 & -1 & +1 & \epsilon \\
\end{array}
\end{pmatrix} 
\overset{2}{\sim}
\begin{pmatrix}
\begin{array}{ccc|c}
 -1 & 0 & 0 & 0 \\
0 & 0 & +1 & -1 \\
 0 & +1 & -1 & +1 \\ \hline
0 & -1 & +1 & \epsilon+1 \\
\end{array}
\end{pmatrix} 
\\ \\
\overset{3}{\sim}
\begin{pmatrix}
\begin{array}{ccc|c}
 -1 & 0 & 0 & 0 \\
0 & -1 & +1 & +1 \\
 0 & +1 & 0 & -1 \\ \hline
0 & +1 & -1 & \epsilon+1 \\
\end{array}
\end{pmatrix} 
&\overset{4}{\sim}
\begin{pmatrix}
\begin{array}{ccc|c}
 -1 & 0 & 0 & 0 \\
0 & -1 & 0 & +1 \\
 0 & 0 & +1 & 0 \\ \hline
0 & +1 & 0 & \epsilon+1 \\
\end{array}
\end{pmatrix} 
\overset{5}{\sim}
\begin{pmatrix}
\begin{array}{ccc|c}
 -1 & 0 & 0 & 0 \\
0 & -1 & 0 & 0 \\
 0 & 0 & +1 & 0 \\ \hline
0 & 0 & 0 & \epsilon+2 \\
\end{array}
\end{pmatrix} 
\end{align*}
\caption{$N_{1,\epsilon} \sim B \oplus L$, where $L$ is the $1\times 1$ matrix $(\epsilon + 2)$. The steps are enumerated, with the step number recorded above the $``\sim"$.}
\label{fig:MatrixBaseCaseExplained}
\end{figure}

\begin{figure}[h!]
\begin{align*}
P_{1,\epsilon} =
\begin{pmatrix}
\begin{array}{ccc|c}
 -1 & +1 & 0 & +1 \\
+1 & -1 & +1 & 0 \\
 0 & +1 & -1 & +1 \\ \hline
+1 & 0 & +1 & \epsilon \\
\end{array}
\end{pmatrix} 
&\overset{1}{\sim}
\begin{pmatrix}
\begin{array}{ccc|c}
 -1 & 0 & 0 & 0 \\
0 & 0 & +1 & +1 \\
 0 & +1 & -1 & +1 \\ \hline
0 & +1 & +1 & \epsilon+1 \\
\end{array}
\end{pmatrix} 
\overset{2}{\sim}
\begin{pmatrix}
\begin{array}{ccc|c}
 -1 & 0 & 0 & 0 \\
0 & -1 & +1 & +1 \\
 0 & +1 & 0 & +1 \\ \hline
0 & +1 & +1 & \epsilon+1 \\
\end{array}
\end{pmatrix} 
\\ \\
\overset{3}{\sim}
\begin{pmatrix}
\begin{array}{ccc|c}
 -1 & 0 & 0 & 0 \\
0 & -1 & 0 & +1 \\
 0 & 0 & +1 & +2 \\ \hline
0 & +1 & +2 & \epsilon+1 \\
\end{array}
\end{pmatrix} 
&\overset{4}{\sim}
\begin{pmatrix}
\begin{array}{ccc|c}
 -1 & 0 & 0 & 0 \\
0 & -1 & 0 & 0 \\
 0 & 0 & +1 & +2 \\ \hline
0 & 0 & +2 & \epsilon+2 \\
\end{array}
\end{pmatrix} 
\overset{5}{\sim}
\begin{pmatrix}
\begin{array}{ccc|c}
 -1 & 0 & 0 & 0 \\
0 & -1 & 0 & 0 \\
 0 & 0 & +1 & 0 \\ \hline
0 & 0 & 0 & \epsilon-2 \\
\end{array}
\end{pmatrix} 
\end{align*}
\caption{$P_{1,\epsilon} \sim B \oplus L$, where $L$ is the $1\times 1$ matrix $(\epsilon - 2)$.}
\label{fig:PConguentD}
\end{figure}

This proves \Cref{claim:PreBaseCase_N}.  
Next, we establish the base cases required to perform strong induction.

\begin{claim}\label{P2}
$P_{2,\epsilon}$ is congruent to a block matrix containing $B$ and $N_{1,\epsilon}$ as factors, as in \Cref{fig:BlockifyP2N2} (left). $N_{2,\epsilon}$ is congruent to a block matrix containing $B$ and $P_{1,\epsilon}$ as factors, as in \Cref{fig:BlockifyP2N2} (right).  
\end{claim}

\begin{figure}[h!]
\begin{align*}
P_{2,\epsilon}
\sim
\begin{pmatrix}
\begin{array}{c|c}
B & \\ \hline
& N_{1,\epsilon}
\end{array}
\end{pmatrix}
\hspace{2cm}
N_{2,\epsilon}
\sim
\begin{pmatrix}
\begin{array}{c|c}
B & \\ \hline
& P_{1,\epsilon}
\end{array}
\end{pmatrix}
\end{align*}
\caption{We claim that $P_{2,\epsilon} \sim B \oplus N_{1,\epsilon}$ and $N_{2,\epsilon} \sim B \oplus P_{1,\epsilon}$.}
\label{fig:BlockifyP2N2}
\end{figure}

\textit{Proof of \Cref{P2}:} 
We first prove the claim for $P_{2,\epsilon}$. We perform a sequence of elementary matrix operations in \Cref{fig:BlockifyP2}. In Step 1, we simultaneously clear the non-zero, non-diagonal entries of the first row and column. In Step 2, we swap the second and third rows and columns. In Step 3, we clear the non-zero, non-diagonal entries of the second row and column. In Step 4, we clear the non-zero, non-diagonal entries of the third row and column. The lower right block of the final matrix is exactly $N_1$. 

The proof for $N_{2,\epsilon}$ is nearly identical, so we suppress the steps for brevity. \hfill $\Box$

\begin{figure}[h!]
\begin{align*}
P_{2,\epsilon} &= 
\begin{pmatrix}
\begin{array}{ccc|ccc|c}
 -1 & +1 & 0 & 0 & 0 & 0 & +1 \\
+1 & -1 & +1 & 0 & 0 & 0 & 0 \\
 0 & +1 & -1 & +1 & 0 & 0 & 0 \\ \hline
 0 & 0 & +1 & -1 & +1 & 0  & 0\\
 0 & 0 & 0 & +1 & -1 & +1 & 0 \\
 0 & 0 & 0 & 0 & +1 & -1 & +1 \\ \hline
+1 & 0 & 0 & 0 & 0 & +1 & \epsilon \\
\end{array}
\end{pmatrix}
\overset{1}{\sim}
\begin{pmatrix}
\begin{array}{ccc|ccc|c}
 -1 & 0 & 0 & 0 & 0 & 0 & 0 \\
0 & 0 & +1 & 0 & 0 & 0 & +1 \\
 0 & +1 & -1 & +1 & 0 & 0 & 0 \\ \hline
 0 & 0 & +1 & -1 & +1 & 0  & 0\\
 0 & 0 & 0 & +1 & -1 & +1 & 0 \\
 0 & 0 & 0 & 0 & +1 & -1 & +1 \\ \hline
0 & +1 & 0 & 0 & 0 & +1 & \epsilon+1 \\
\end{array}
\end{pmatrix}
\\ \\
&\overset{2}{\sim}
\begin{pmatrix}
\begin{array}{ccc|ccc|c}
 -1 & 0 & 0 & 0 & 0 & 0 & 0 \\
0 & -1 & +1 & +1 & 0 & 0 & 0 \\
 0 & +1 & 0 & 0 & 0 & 0 & +1 \\ \hline
 0 & +1 & 0 & -1 & +1 & 0  & 0\\
 0 & 0 & 0 & +1 & -1 & +1 & 0 \\
 0 & 0 & 0 & 0 & +1 & -1 & +1 \\ \hline
0 & 0 & +1 & 0 & 0 & +1 & \epsilon+1 \\
\end{array}
\end{pmatrix}
\overset{3}{\sim}
\begin{pmatrix}
\begin{array}{ccc|ccc|c}
 -1 & 0 & 0 & 0 & 0 & 0 & 0 \\
0 & -1 & 0 & 0 & 0 & 0 & 0 \\
 0 & 0 & +1 & +1 & 0 & 0 & +1 \\ \hline
 0 & 0 & +1 & 0 & +1 & 0  & 0\\
 0 & 0 & 0 & +1 & -1 & +1 & 0 \\
 0 & 0 & 0 & 0 & +1 & -1 & +1 \\ \hline
0 & 0 & +1 & 0 & 0 & +1 & \epsilon+1 \\
\end{array}
\end{pmatrix}
\\ \\
&\overset{4}{\sim}
\begin{pmatrix}
\begin{array}{ccc|ccc|c}
 -1 & 0 & 0 & 0 & 0 & 0 & 0 \\
0 & -1 & 0 & 0 & 0 & 0 & 0 \\
 0 & 0 & +1 & 0 & 0 & 0 & 0 \\
 \hline
 0 & 0 & 0 & -1 & +1 & 0  & -1\\
 0 & 0 & 0 & +1 & -1 & +1 & 0 \\
 0 & 0 & 0 & 0 & +1 & -1 & +1 \\
 \hline
0 & 0 & 0 & -1 & 0 & +1 & \epsilon \\
\end{array}
\end{pmatrix}
=
\begin{pmatrix}
\begin{array}{c|c}
B & \\ \hline
& N_{1,\epsilon}
\end{array}
\end{pmatrix}
\end{align*}
\caption{$P_{2,\epsilon}$ is a block diagonal matrix with $P_{2,\epsilon} \sim B \oplus N_{1,\epsilon}$.}
\label{fig:BlockifyP2}
\end{figure}

\Cref{claim:PreBaseCase_N} and \Cref{P2} establish the base cases for our induction. To prove the lemma, we will use strong induction: we assume that for all $2 \leq i \leq k$, we have that $P_{i,\epsilon}$ and $N_{i, \epsilon}$ are congruent to block diagonal matrices of the form in \Cref{fig:InductionHypothesis}. Therefore, to prove the claim, it suffices to establish the inductive step.

\begin{figure}[h!]
\begin{align*}
P_{i,\epsilon}
\sim
\begin{pmatrix}
\begin{array}{c|c}
B & \\ \hline
& N_{i-1,\epsilon}
\end{array}
\end{pmatrix}
\hspace{2cm}
N_{i,\epsilon}
\sim
\begin{pmatrix}
\begin{array}{c|c}
B & \\ \hline
& P_{i-1,\epsilon}
\end{array}
\end{pmatrix}
\end{align*}
\caption{We claim $P_{i,\epsilon} \sim B \oplus N_{i-1,\epsilon}$ and $N_{i,\epsilon} \sim B \oplus P_{i-1,\epsilon}$.}
\label{fig:InductionHypothesis}
\end{figure}


\begin{claim} \label{InductiveStep}
For $k\geq 2$, the matrix 
$P_{k+1,\epsilon}$ (resp. $N_{k+1,\epsilon}$) is congruent to the block diagonal matrix with factors $B$ and $N_{k,\epsilon}$ (resp. $P_{k,\epsilon}$). 
\end{claim}

\textit{Proof of \Cref{InductiveStep}:} We prove the claim for $P_{k+1, \epsilon}$; the steps are in \Cref{fig:BlockifyPk}. We entries that stay unchanged throughout the process are labelled with a "$\star$". The suppressed block with the "$\star$" entries is a $((3k-1) \times (3k-1))$ matrix. In Step 1, we simultaneously clear the non-zero, non-diagonal entries of the first row/column. In Step 2, we swap the second and third rows and columns. In Step 3, we clear the non-zero, non-diagonal entries of the second row/column. In Step 4, we clear the non-zero, non-diagonal entries of the third row/column. Finally, we redraw our guidelines to see that the resulting matrix contains $B$ and $N_{k,\epsilon}$ as factors. 

The analogous manipulation on $N_{k+1,\epsilon}$ yields the result. 
 \hfill $\Box$

\begin{figure}[h!]
\begin{align*}
P_{k+1,\epsilon} &= 
\begin{pmatrix}
\begin{array}{cccc|ccc|c}
 -1 & +1 &  &  &  &  & & +1 \\
+1 & -1 & +1 &  &  &  &  \\
  & +1 & -1 & +1 &  &  &  \\ 
  &  & +1 & -1 & +1 &   & \\ \hline
  &  &  & +1 & \star & \star & \star \\
   &  &  &  & \star & \star & \star \\ 
  &  &  &  & \star & \star & \star&  +1 \\ \hline
+1 &  &  &  & &  & +1&  \epsilon \\
\end{array}
\end{pmatrix}
\overset{1}{\sim}
\begin{pmatrix}
\begin{array}{cccc|ccc|c}
 -1 &  &  &  &  &  & &  \\
 &  & +1 &  &  &  & &+1  \\
  & +1 & -1 & +1 &  &  &  \\ 
  &  & +1 & -1 & +1 &   & \\ \hline
  &  &  & +1 & \star & \star & \star \\
   &  &  &  & \star & \star & \star \\ 
  &  &  &  & \star & \star & \star&  +1 \\ \hline
 &  +1 &  &  & &  & +1&  \epsilon + 1 \\
\end{array}
\end{pmatrix}
\\ \\
&\overset{2}{\sim}
\begin{pmatrix}
\begin{array}{cccc|ccc|c}
 -1 &  &  &  &  &  & & \\
 & -1 & +1 & +1 &  &  & & \\
  & +1 &  &  &  &  & & +1 \\ 
  & +1 &  & -1 & +1 &   & \\ \hline
 &  &  & +1 & \star & \star & \star \\ 
  &  &  &  & \star & \star & \star \\
  &  &  &  & \star & \star & \star & +1\\ \hline
 &  & +1 &  &  &  & +1 & \epsilon+1 \\
\end{array}
\end{pmatrix}
\overset{3}{\sim}
\begin{pmatrix}
\begin{array}{cccc|ccc|c}
 -1 &  &  &  &  &  & & \\
 & -1 &  &  &  &  &  &  \\
  &  & +1 & +1 & & & & +1 \\ 
  &  & +1 &  & +1 &   & \\ \hline
  &  &  &  +1 & \star & \star & \star \\
    &  &  &  & \star & \star & \star \\
  &  &  &  & \star & \star & \star & +1\\ \hline
 &  & +1 &  &  &  &+1 & \epsilon+1 \\
\end{array}
\end{pmatrix}
\\ \\
&\overset{4}{\sim}
\begin{pmatrix}
\begin{array}{cccc|ccc|c}
 -1 &  &  &  &  &  & & \\
 & -1 &  &  &  &  &  &  \\
  &  & +1 &  & & & & \\ 
  &  &  &  -1& +1 & & & -1 \\ \hline
  &  &  &  +1 & \star & \star & \star \\
    &  &  &  & \star & \star & \star \\
  &  &  &  & \star & \star & \star & +1 \\ \hline
 &  &  & -1 &  &  &+1 & \epsilon \\
\end{array}
\end{pmatrix}
=
\begin{pmatrix}
\begin{array}{ccc|cccc|c}
 -1 &  &  &  &  &  & & \\
 & -1 &  &  &  &  &  &  \\
  &  & +1 &  & & & & \\ \hline
  &  &  &  -1& +1 & & & -1 \\ 
  &  &  &  +1 & \star & \star & \star \\
    &  &  &  & \star & \star & \star \\
  &  &  &  & \star & \star & \star & +1 \\ \hline
 &  &  & -1 &  &  &+1 & \epsilon \\
\end{array}
\end{pmatrix}
\end{align*}
\caption{Proving that $P_{i,\epsilon} \sim B \oplus N_{i-1,\epsilon}$.}
\label{fig:BlockifyPk}
\end{figure}

We can now finish the proof of \Cref{lemma:BigLemma}: \textit{for all $q \geq 1$, $P_{q, \epsilon}$ is congruent to a diagonal matrix with $q$-many $B$ factors, and a single $1 \times 1$ factor with entry $\epsilon \pm 2$.} We begin by diagonalizing the first three rows and columns of $P_{q,\epsilon}$ to produce a block diagonal matrix with factors $B$ and $N_{q-1, \epsilon}$. Then, we factor $N_{q-1, \epsilon}$ into a block diagonal matrix with factors $B$ and $P_{q-2, \epsilon}$. Exhausting this procedure eventually yields a factorization with $q$-many $B$ blocks. The remaining block is a $1 \times 1$ matrix. When $q$ is even, this entry will be $\epsilon+2$, and when $q$ is odd, the entry will be $\epsilon-2$. 
\end{proof}

This concludes the proof of \Cref{lemma:BigLemma}. We can now prove \Cref{DiagonalizeG(K)}, which claimed that $G(K)$ can be diagonalized. 

\begin{proof}[Proof of \Cref{DiagonalizeG(K)}.] Let $K$ be the specified knot, and construct a spanning surface and Goeritz matrix for $K$ as in \Cref{GoeritzSetup}. Notice that $G(K) = P_{k,2m-1}$. By \Cref{lemma:BigLemma}, $P_{k,2m-1}$ is congruent to a diagonal matrix $D$, where $D = (\bigoplus_{i=1}^k B) \oplus L$, where $L$ is a $1 \times 1$ matrix with entry $(2m-1) \pm 2$. This is what we wanted to show. 
\end{proof}

\begin{thm} \label{thm:Signature}
Let $K = T(3,3k+1;2m)$ where $k\geq 1$ and $m \geq 0$. Then 
\begin{align*}
\sigma(K) = \begin{cases}
-4k-2m-2 & \qquad k \equiv 1 \mod 2, \ \ m = 0, 1 \\
-4k-2m & \qquad \text{otherwise}
\end{cases}
\end{align*}
\end{thm}

\begin{proof}
For the specified $K$, \Cref{DiagonalizeG(K)} tells us that the Goeritz matrix $G(K)$ is congruent to $D = (\bigoplus_{i=1}^k B) \oplus L$, where $L$ is a $1 \times 1$ matrix with entry $(2m-1) \pm 2$. The entry of $L$ is completely determined by $m$ and the parity of $k$. In particular, $(2m-1) \pm 2 < 0$ if and only if $m = 0, 1$ and $k$ is odd. Therefore, by applying Sylvester's theorem of inertia, we have:
\begin{align} \label{eqn:SignG(K)}
\sgn(G(K))= \sgn(P_{k,2m-1}) = \sgn(D) = \begin{cases}
-k-1 & \qquad k \equiv 1 \mod 2, \ \ m = 0, 1 \\
-k+1 & \qquad \text{otherwise}
\end{cases}
\end{align}
By \Cref{GordonLitherland_Mu}, $\mu(K) = (3k+1)+2m$. Combining this with $\sigma(K) = \sgn(G(K))-\mu(K)$, we deduce the result.
\end{proof}

\begin{rmk}
Suppose $k \geq 1$ and $m \geq 0$. If $\sigma(T(3,3k+1;2m)) = -4k-2m-2$, then by \Cref{thm:Signature}, $K$ is a torus knot: either $K \approx T(3, 3k+1)$ or $K \approx T(3,3k+1; 2)$, which, by \Cref{lemma:AdjacentTTKs}, is isotopic to $T(3, 3k+2)$. 
\end{rmk}

\subsection{Signatures and the Seifert genus}

The following lemmas study how the signatures of twisted torus knots of a fixed genus are related.


\begin{lemma} \label{lemma:DifferentSignaturesCase1}
Let $K_1 = T(3,3k+1;2m)$ and $K_2=T(3,3(k-1)+1; 2(m+3))$ where $m \geq 2$ and $(3(k-1)+1) \geq 4$. Then $g(K_1)=g(K_2)$, but $\sigma(K_1) - \sigma(K_2) = 2$.
\end{lemma}

\begin{proof}
As in the proof of \Cref{lemma:PotentialConcordance}, to prove that $g(K_1) = g(K_2)$, it suffices to check that $K_1$ and $K_2$ have the same writhe:
\[
wr(K_1) = 2(3k+1)+2m = 2(3k+1 + 3 - 3) + 2m = 2(3(k-1)+1) + 6 + 2m = wr(K_2)
\]

$K_1$ and $K_2$ both satisfy the hypotheses of \Cref{thm:Signature}. Since we assumed that $m \geq 2$, the signatures of $K_1$ and $K_2$ have forms in the second case of \Cref{thm:Signature}. Thus, 
\begin{align*}
\sigma(K_1) - \sigma(K_2) = (-4k-2m) - (-4(k-1)-2(m+3)) =2
\end{align*}
This is what we wanted to show.
\end{proof}

\begin{lemma} \label{lemma:DifferentSignaturesCase2}
Suppose $k \geq 1$. For $0 \leq i \leq k-1$, let $K_i = T(3, 3(k-i)+1; 6i)$. The knots $K_i$ have the same Seifert genus. Additionally, if $k \equiv 0 \mod 2$, then for all $i$, $\sigma(K_i) - \sigma(K_{i+1})=2$, but if $k \equiv 1 \mod 2$, then $\sigma(K_0) = \sigma(K_1)$, and for all $i \geq 1$, $\sigma(K_i) - \sigma(K_{i+1})=2$. 
\end{lemma}

\begin{proof}
It is elementary to check that all the knots have the same writhe. Since they are all positive 3-braid closures, it follows they have the same Seifert genus. 

Suppose $k$ is even. By applying the \Cref{thm:Signature}, we see that $\sigma(K_0) = -4k$, $\sigma(K_1) = -4k-2$, and $\sigma(K_2)=-4k-4$. In particular, $\sigma(K_0) - \sigma(K_1) = 2$ and $\sigma(K_1) - \sigma(K_2) = 2$. Applying \Cref{lemma:DifferentSignaturesCase1} for $2 \leq i \leq s-1$ yields the desired result. 

Now suppose $k$ is odd. \Cref{thm:Signature} tells us that $\sigma(K_0) = -4k-2$, $\sigma(K_1) = -4k-2$, and $\sigma(K_2)=-4k-4$. In particular, $\sigma(K_0) = \sigma(K_1)$ and $\sigma(K_1) - \sigma(K_2) = 2$. Applying \Cref{lemma:DifferentSignaturesCase1} for $2 \leq i \leq s-1$ concludes the proof.
\end{proof}

\begin{lemma} \label{lemma:DifferentSignaturesCase3}
Suppose $k \geq 1$. For $0 \leq i \leq k-1$, let $K_i = T(3, 3(k-i)+1; 2+6i)$. The knots $K_i$ have the same Seifert genus. Additionally, if $k \equiv 0 \mod 2$, then for all $i$, $\sigma(K_i) - \sigma(K_{i+1})=2$, but if $k \equiv 1 \mod 2$, then $\sigma(K_0) = \sigma(K_1)$, and for all $i \geq 1$, $\sigma(K_i) - \sigma(K_{i+1})=2$. 
\end{lemma}

\begin{proof}
The proof is nearly identical to the one in \Cref{lemma:DifferentSignaturesCase2}. Note that if $k \equiv 1 \mod 2$, then $\sigma(K_0)=\sigma(K_1)=-4k-4$. 
\end{proof}

To prove \Cref{thm:DistinctConcordanceClasses}, we need one last lemma which analyzes the case where a twisted torus knot and a torus knot have the same Seifert genus and signature. 

\begin{lemma} \label{lemma:DifferentSignaturesCase4}
Suppose $K_1 = T(3,q)$ is a torus knot, and $K_2$ is a twisted torus knot on three strands such that $g(K_1) = g(K_2)$ and $\sigma(K_1) = \sigma(K_2)$. The knots $K_1$ and $K_2$ are not concordant.
\end{lemma}

\begin{proof}
Suppose $K_1$ and $K_2$ are a pair of knots satisfying the hypothesis of the lemma. By \Cref{lemma:DifferentSignaturesCase2} and \Cref{lemma:DifferentSignaturesCase3}, we must be in one of the following two scenarios:

\begin{enumerate}
\item $k$ is odd, with $K_1 = T(3, 3k+1; 0)$ and $K_2=T(3, 3(k-1)+1; 6)$. 
\item $k$ is odd, with $K_1 = T(3, 3k+1; 2)$ and $K_2=T(3, 3(k-1)+1; 8)$.
\end{enumerate}

We first prove the lemma for \textbf{Case (1)}. If $K_1$ is concordant to $K_2$, then $K = K_1 \# m(K_2^r)$ is slice. The Fox-Milnor criterion says that the Alexander polynomial of $K$ would factor in a particular way: in particular, $\Delta_K(t)=f(t)f(t^{-1})$. To obstruct a concordance between $K_1$ and $K_2$, we will show that the putative factorization of Alexander polynomial does not exist.

We recall that torus knots have simple Alexander polynomials \cite{Rolfson:KnotsAndLinks}: in general, 
\begin{align*}
\Delta_{T(p,q)}(t) = \displaystyle \prod _{\substack{h|p, \ell | q \\ h, \ell \neq 1}} \phi_{\ell h}(t) 
\end{align*}
where $\phi_n(t)$ denotes the $n$-th cyclotomic polynomial. Thus, 
$\Delta_{K_1}(t) = \Delta_{T(3,q)}(t) = \displaystyle \prod _{\substack{\ell | q, \ \ell \neq 1}} \phi_{3\ell}(t)$. 

In particular, $\Delta_{K_1}(t)$ factors into a product of cyclotomic polynomials, where there are no repeated factors. Therefore, to prove that $\Delta_K(t)$ cannot factor as described by the Fox-Milnor criterion, it suffices to prove that $\Delta_{K_2}(t) \neq \Delta_{K_1}(t)$. We will use the reduced Burau matrix, which we described in \Cref{AlexPolyResult}. In particular, when a knot is realized as the closure of some braid $B$, we have
$$\frac{\det(I_{n-1}-B(t))}{1-t^n} = \frac{\Delta_{\widehat{\beta}}(t)}{1-t}$$
$K_1$ and $K_2$ are realized on $n=3$ strands. So, to show that $\Delta_{K_1}(t) \neq \Delta_{K_2}(t)$, we must prove that: 
\begin{align}
\det(I_2 - B_1(t)) \neq \det(I_2 - B_2(t)) \label{DifferentAlexPoly}
\end{align}

where $B_1(t)$ and $B_2(t)$ denote the reduced Burau matrices of $\beta_1$ and $\beta_2$, respectively, where $\widehat{\beta_i} \approx K_i$. Additionally, we notice that $\det(I_2 - B_i(t)) = q_{B_i(t)}(1)=1-\tr(B_i(t))+\det(B_i(t))$, where $q_{B(t)}(x)$ is the characteristic-type polynomial defined in \Cref{eqn:Burau}. Since $\beta_1$ and $\beta_2$ are both positive 3-braids with the same writhe, $\det(B_1(t)) = \det(B_2(t))$. Therefore, to deduce that (\ref{DifferentAlexPoly}) holds, it suffices to prove that $\tr(B_1(t)) \neq \tr(B_2(t))$. 

Recall that 
\begin{align*}
\sigma_1(t) &= 
\begin{pmatrix}
\begin{array}{cc}
-t & 1 \\
0 & 1
\end{array}
\end{pmatrix}
\hspace{2cm} \text{ and } \hspace{2cm}
\sigma_2(t) = 
\begin{pmatrix}
\begin{array}{cc}
1 & 0 \\
t & -t
\end{array}
\end{pmatrix}
\end{align*}
Therefore, we know:
\begin{align*}
\sigma_2(t) \sigma_1(t) =
\begin{pmatrix}
\begin{array}{cc}
-t & 1 \\
-t^2 & 0
\end{array}
\end{pmatrix}
\text{, } \hspace{1em}
(\sigma_2(t) \sigma_1(t))^3 =
\begin{pmatrix}
\begin{array}{cc}
t^3 & 0 \\
0& t^3
\end{array}
\end{pmatrix}
\text{, \ \ } \hspace{0.5em}
(\sigma_1)^6 =
\begin{pmatrix}
\begin{array}{cc}
t^6 & -t^5+t^4-t^3+t^2-t+1 \\
0& 1
\end{array}
\end{pmatrix}
\end{align*}

This allows us to compute $B_1(t)$ and $B_2(t)$. 
\begin{align*}
B_1(t) = (\sigma_2(t) \sigma_1(t))^{3k+1} 
&= \begin{pmatrix}
\begin{array}{cc}
t^3 & 0 \\
0& t^3
\end{array}
\end{pmatrix}^k \cdot 
\begin{pmatrix}
\begin{array}{cc}
-t & 1 \\
-t^2 & 0
\end{array}
\end{pmatrix} 
=
\begin{pmatrix}
\begin{array}{cc}
-t^{3k+1} & t^{3k} \\
-t^{3k+2} & 0
\end{array}
\end{pmatrix} \\ \\
%
%
%
B_2(t) = (\sigma_2(t) \sigma_1(t))^{3(k-1)+1} (\sigma_1)^6
&= \begin{pmatrix}
\begin{array}{cc}
t^3 & 0 \\
0& t^3
\end{array}
\end{pmatrix}^{k-1} \cdot 
\begin{pmatrix}
\begin{array}{cc}
-t & 1 \\
-t^2 & 0
\end{array}
\end{pmatrix} \cdot
\begin{pmatrix}
\begin{array}{cc}
t^6 & -t^5+t^4-t^3+t^2-t+1 \\
0 & 1
\end{array}
\end{pmatrix} \\
&= 
t^{3(k-1)}
\cdot 
\begin{pmatrix}
\begin{array}{cc}
-t^7 & t^6-t^5+t^4-t^3+t^2-t+1 \\
-t^8 & t^7-t^6+t^5-t^4+t^3-t^2
\end{array}
\end{pmatrix} \\
%
\end{align*}
It is immediate that $\tr(B_1(t)) \neq \tr(B_2(t))$. Therefore, we see that the quantities $\det(I_2 - B_1(t))$ and $\det(I_2-B_2(t))$ do not agree, thus $\Delta_{K_1}(t) \neq \Delta_{K_2}(t)$, and $\Delta_K(t)$ cannot factor as $f(t)f(t^{-1})$. 
Therefore, in \textbf{Case (1)}, the knots $K_1$ and $K_2$ cannot be smoothly concordant.

An analogous argument works for \textbf{Case (2)}, which we verify below. Recall that, by \Cref{lemma:AdjacentTTKs}, $T(3, 3k+1;2) \approx T(3, 3k+2)$. So, in this case, we have:
\begin{align*}
B_1(t) &= (\sigma_2(t) \sigma_1(t))^{3k+2} 
=
t^{3k}
\begin{pmatrix}
\begin{array}{cc}
-t & 1 \\
-t^2 & 0
\end{array}
\end{pmatrix} ^2
=
\begin{pmatrix}
\begin{array}{cc}
0 & -t^{3k+1} \\
t^{3k+3} & -t^{3k+2}
\end{array}
\end{pmatrix} \\  \\
B_2(t) &= (\sigma_2(t) \sigma_1(t))^{3(k-1)+1} (\sigma_1)^8
=
t^{3(k-1)} \cdot
\begin{pmatrix}
\begin{array}{cc}
-t & 1 \\
-t^2 & 0
\end{array}
\end{pmatrix} \cdot
\begin{pmatrix}
\begin{array}{cc}
t^8 & -t^7+t^6-t^5+t^4-t^3+t^2-t+1 \\
0 & 1
\end{array}
\end{pmatrix}  \\ 
\end{align*}

Once again, the quantities $\tr(B_1(t))$ and $\tr(B_2(t))$ cannot agree, so neither can the corresponding Alexander polynomials, and the knots cannot be concordant.
\end{proof}

We can now prove our last main result.

\textbf{\Cref{thm:DistinctConcordanceClasses}} \textit{
Let $\mathcal{S} = \mathcal{T} \cup \mathcal{L}$, where $\mathcal{T}$ is the set of all positive torus knots, and $\mathcal{L}$ is the set of L-space knots of braid index three. Every concordance contains at most one knot from $\mathcal{S}$.}


As we will see, the knot invariants $i(K)$, $\tau(K)$, and $\sigma(K)$, when combined, are sufficient to show that every knot in $\mathcal{S}$ lies in a distinct concordance class.

\begin{proof}
Let $\mathcal{S}$ be the the set of knots stated above. Every knot $K \in \mathcal{S}$ is a twist positive L-space knot, so by \Cref{thm:Lorenz}, they each realize the minimal braid index of their concordance classes. Additionally, notice that $\mathcal{T} \cap \mathcal{L} = \{T(3,q)\ |\ q \geq 4\}$. 

First, we notice that if the positive torus knots $T_1$ and $T_2$ are concordant to each other, then by \Cref{thm:Lorenz}, they have the same braid index. Therefore, $T_1=T(p,q_1)$ and $T_2=T(p,q_2)$ for some $p \geq 2, q_i \geq 3$. However, if they are concordant, then $\tau(T_1) = \tau(T_2)$. Since $T_1$ and $T_2$ are fibered, strongly quasipositive knots, then by \cite{Hedden:Positivity}, $\tau(T_1)=g(T_1)$ and $\tau(T_2)=g(T_2)$. 
Since $g(T(r,s)) = \frac{(r-1)(s-1)}{2}$, we deduce that $q_1 = q_2$, hence $T_1 \approx T_2$. Therefore, all the elements of $\mathcal{T}$ represent distinct concordance classes. 

It remains to study how the elements of $\mathcal{L}$ interact with each other, and the elements of $\mathcal{T}$. If $K \in \mathcal{L}$, then by \Cref{thm:Lorenz}, the only knots in $\mathcal{S}$ to which $K$ could be concordant must also have braid index 3, so $K$ could only be concordant to another knot in $\mathcal{L}$. Thus, it suffices to show that every $K \in \mathcal{L}$ represents a distinct concordance class.

We begin by sorting all knots in $\mathcal{L}$ by their $\tau$ invariants. Since all knots under consideration are fibered and strongly quasipositive, we can instead recast this as a sorting by Seifert genus. For all $g \geq 3$, there exists a unique torus knot $T(3,q)$ of genus $g$.  As we proved in \Cref{lemma:PotentialConcordance}, there will be $s = \lfloor \frac{q}{3} \rfloor - 1$ twisted torus knots which could be potentially concordant to it. We enumerate these knots as follows: $K_0 = T(3,q), \ldots, K_{s} \in \{T(3,4; 2(3s)), T(3,4; 2(3s+1))\}$, where (as in as in Lemmas \ref{lemma:DifferentSignaturesCase2} and \ref{lemma:DifferentSignaturesCase3}) $K_{s}$ is determined by the congruence class of $q \mod 3$.

By combining Lemmas \ref{lemma:DifferentSignaturesCase1}, \ref{lemma:DifferentSignaturesCase2}, \ref{lemma:DifferentSignaturesCase3}, we see that all the knots potentially concordant to $K_0$ all have pairwise distinct signatures (hence they must represent pairwise distinct concordance classes), or there is exactly one pair of knots with the same signature; however, if the latter occurs, we can apply \Cref{lemma:DifferentSignaturesCase4}, and deduce that none of the knots with the same $\tau$ invariant as $K_0$ can be concordant to it. Thus, all the knots $K_0, K_1, \ldots, K_s$ must lie in distinct concordance classes.

Therefore, every knot $K \in L$ lies in a unique concordance class, and these must be distinct from the classes represented by $\mathcal{T}$. Thus, every element of $\mathcal{S}$ represents a distinct concordance class.
\end{proof}

\textbf{\Cref{cor:InfinitelyDistinctConcordance}.}
\textit{There is an explicit infinite family of positive braid knots that are distinct in concordance, where as $g \to \infty$, the number of \textit{hyperbolic} knots of genus $g$ gets arbitrarily large.}

\begin{proof}
The explicit infinite family of positive braid knots is the family $\mathcal{S}$ from \Cref{thm:DistinctConcordanceClasses}. 
The statement follows by combining \Cref{thm:DistinctConcordanceClasses}, \Cref{lemma:hyperbolic}, and \Cref{lemma:PotentialConcordance}.
\end{proof}

\section{Discussion} \label{Discussion}

In this section, we justify some of the claims we made in the introduction. Namely, we provide evidence for \Cref{LspaceIndex}: \textit{for any L-space knot, the braid index and bridge index agree}. 

\begin{prop}
Let $K_n, n \geq 1$ be the infinite family of hyperbolic L-space knots defined by Baker-Kegel \cite[Section 2]{BakerKegel}. Then, for all $n$, $br(K_n)=i(K_n) = 4$. 
\end{prop}

\begin{proof}
Baker-Kegel define the knot $K_n$ to be to be the closure of the 4-braid $\beta_n$, where $$\beta_n = (\sigma_2 \sigma_1 \sigma_3 \sigma_2)^{2n+1} {\sigma_1}^{-1} \sigma_2 \sigma_1 \sigma_1 \sigma_2$$

In particular, for all $n\geq 1$, $i(K_n) \leq 4$. By inspecting \cite[Proposition 3.1 (3)]{BakerKegel}, we see that 
$$\Delta_{K_n}(t) = 1-t+t^4+t^5 + \ldots + t^{4n+2} + \ldots -t^{8n+3}+t^{8n+2}$$
In particular, we see that 3 appears as a gap in the exponents of the Alexander polynomial of these knots. Thus, as in the proof of \Cref{thm:Lorenz}, $$3 \leq Ord_v(K_n) \leq br(K_n)-1 \leq i(K_n)-1 \leq 3$$
In particular, we must have equalities throughout, and $br(K_n)=i(K_n)=4$. 
\end{proof}



\begin{prop} \label{lemma:ConjForCables}
If \Cref{LspaceIndex} holds for an L-space knot $K$, then it holds for all of its cables.
\end{prop}

\begin{proof}
Schultens \cite{Schultens:BridgeIndex} reproves a classical result of Schubert \cite{Schubert:BridgeIndex}: suppose $S$ is a satellite knot with companion $J$ and pattern $P$ of wrapping number $k$ (the wrapping number is the minimal geometric intersection number of the pattern with a meridional disk of the pattern torus), then $br(S)\geq k\cdot br(J)$. In particular, if $K_{p,q}$ is a cable of $K$, then $br(K_{p,q}) \geq p \cdot br(K)$.  Williams \cite{Williams:Cables} proved that analogous result holds the braid index: $i(K_{p,q})=p\cdot i(K)$. Therefore, we have:
\begin{align} \label{eqn:Conjecture}
p \cdot br(K) \leq br(K_{p,q}) \leq i(K_{p,q}) = p\cdot i(K)
\end{align}

If $K$ is an L-space knot with $i(K)=br(K)$, then we have equalities throughout \Cref{eqn:Conjecture}, proving the claim. 
\end{proof}

We hope these observations provide some motivation for investigating \Cref{LspaceIndex} further.

\bibliographystyle{amsalpha2}
\bibliography{../../../masterbiblio}

\end{document}